\documentclass[preprint,11pt]{elsarticle}

\usepackage[a4paper,margin=1in]{geometry}
\journal{Submitted for publication}
\biboptions{numbers,sort&compress}
\usepackage{amsmath,amssymb,amsthm,mathtools,bm,mathrsfs}
\usepackage{microtype}
\usepackage{enumitem}
\usepackage{booktabs}
\usepackage{xurl}

\usepackage[colorlinks=true,linkcolor=blue,citecolor=blue,urlcolor=blue]{hyperref}
\hypersetup{
  pdftitle={Double-logarithmic stability for a parabolic inverse problem on CTA manifolds},
  pdfauthor={Yujian Zheng, Zhiwen Duan, Shiqi Jing}
}

\numberwithin{equation}{section}

\newtheorem{theorem}{Theorem}[section]
\newtheorem{proposition}[theorem]{Proposition}
\newtheorem{lemma}[theorem]{Lemma}
\newtheorem{corollary}[theorem]{Corollary}
\newtheorem{assumption}{Assumption}
\theoremstyle{definition}
\newtheorem{definition}{Definition}[section]
\theoremstyle{remark}
\newtheorem{remark}[theorem]{Remark}

\newcommand{\Q}{\mathcal Q}
\newcommand{\U}{{\Sigma_{\mathrm D}}}
\newcommand{\V}{{\Sigma_{\mathrm N}}}
\newcommand{\D}{\mathcal D}
\newcommand{\X}{\mathcal X}
\newcommand{\Y}{\mathcal Y}
\newcommand{\G}{\mathcal G}
\newcommand{\Hc}{\mathcal H}

\newcommand{\supp}{\operatorname{supp}}

\newcommand{\dd}{\,\mathrm d}
\newcommand{\eps}{\varepsilon}
\newcommand{\ip}[2]{\left\langle #1,#2\right\rangle}

\begin{document}
\begin{frontmatter}
\title{Double-logarithmic stability for a parabolic inverse problem on CTA manifolds}
\author[math]{Yujian Zheng}
\ead{d202280014@hust.edu.cn}
\author[math,lab]{Zhiwen Duan}
\ead{duanzhw@hust.edu.cn}
\author[math]{Shiqi Jing\corref{cor1}}
\ead{d202580010@hust.edu.cn}
\cortext[cor1]{Corresponding author.}
\address[math]{School of Mathematics and Statistics, Huazhong University of
Science and Technology, Wuhan 430074, Hubei, China}
\address[lab]{Hubei Key Laboratory of Engineering Modeling and Scientific
Computing, Huazhong University of Science and Technology, Wuhan 430074, Hubei, China}
\begin{abstract}
We establish a double-logarithmic stability estimate for the recovery of
a time-dependent potential in the parabolic equation
$(c^{-1}\partial_t-\Delta_g+q)u=0$ on a conformally transversally
anisotropic manifold. The partial input--output map associates an initial
state and lateral Dirichlet data supported in a prescribed open set with
the terminal state and the Neumann trace on another open set. These sets
contain the positive and negative boundary parts determined by a limiting
Carleman weight, respectively. The measurement discrepancy is the operator
norm of the difference of two such maps, whose range has improved Sobolev
regularity. We prove stability under a uniform stability assumption for
the attenuated geodesic ray transform on a family of non-tangential
geodesics. This assumption is verified when the transversal manifold is
simple and for regular families on certain non-simple manifolds. The proof
uses Gaussian beam quasimodes with a uniform $O(h^{1/2})$ concentration
estimate, geometric optics solutions obtained from boundary Carleman
estimates, and stability of the geodesic ray transform for small
attenuation. Quantitative analytic continuation for Hilbert-space-valued
Fourier transforms and Sobolev interpolation yield the double-logarithmic
modulus.
\end{abstract}
\end{frontmatter}

\section{Introduction and main results}\label{sec:intro}
\enlargethispage{2pt}

Let $(M,g)$ be a smooth compact oriented Riemannian manifold of dimension
$n\geq3$ with smooth boundary. We assume that $(M,g)$ is
\emph{conformally transversally anisotropic} (CTA), in the sense that
\begin{equation}\label{eq:cta}
 M\Subset\mathbb R\times M_0^{\mathrm{int}},
 \qquad g=c(x)(\dd x_1^2\oplus g_0),
\end{equation}
where $(M_0,g_0)$ is a smooth compact $(n-1)$-dimensional manifold
with boundary, called the \emph{transversal manifold}, and $c$ is a known
smooth positive conformal factor. For $T>0$, denote the
\emph{space-time cylinder} and its \emph{lateral boundary} by
\[
 Q=(0,T)\times M^{\mathrm{int}},\qquad
 \Sigma=(0,T)\times\partial M,
\]
respectively. We consider the parabolic initial-boundary value problem
\begin{equation}\label{eq:ibvp}
\begin{cases}
 (c^{-1}\partial_t-\Delta_g+q(t,x))u=0&\text{in }Q,\\
 u=f&\text{on }\Sigma,\\
 u(0,\cdot)=u_0&\text{in }M.
\end{cases}
\end{equation}
The factor $c^{-1}$ in front of the time derivative is part of the model.  It
allows an exact conformal reduction to the product metric, changing only the
zeroth-order coefficient; see Section~\ref{sec:framework}.  Thus the time coefficient is chosen so that conformal reduction gives
the heat operator with a potential on the product manifold.  Unless \(c=1\), the result is not a statement
for the ordinary operator \(\partial_t-\Delta_g+q\) with the same metric.
The geometry in \eqref{eq:cta} is the natural manifold setting associated
with the limiting Carleman weight \(x_1\); see
\cite{DSFKSU2009,KenigSalo2013}.  Gaussian beam constructions permit the
simple-transversal assumption used in admissible geometries to be replaced by
information along a suitable family of non-tangential geodesics
\cite{DSFKLS2016}.

The limiting Carleman weight is $\varphi(x)=x_1$.  With $\nu$ denoting the
outward unit normal, define the closed positive and negative boundary parts
\[
 \partial M_+=\{x\in\partial M:\partial_\nu\varphi(x)\geq0\},
 \qquad
 \partial M_-=\{x\in\partial M:\partial_\nu\varphi(x)\leq0\}.
\]
Their intersection is the \emph{tangential boundary} (or glancing set),
\[
 \partial M_{\mathrm{tan}}
 =\{x\in\partial M:\partial_\nu\varphi(x)=0\}
 =\partial M_+\cap\partial M_-.
\]
Our notation $\partial M_\pm$ includes this zero set; their strict-sign
parts are $\partial M_\pm\setminus\partial M_{\mathrm{tan}}$.
Let $\Gamma_{\mathrm D},\Gamma_{\mathrm N}\subset\partial M$ be the
open \emph{Dirichlet input set} and \emph{Neumann observation set},
with smooth relative boundaries, satisfying
\begin{equation}\label{eq:boundary-neighborhoods}
 \partial M_+\Subset \Gamma_{\mathrm D},
 \qquad \partial M_-\Subset \Gamma_{\mathrm N},
 \qquad
 \U=(0,T)\times \Gamma_{\mathrm D},\quad \V=(0,T)\times \Gamma_{\mathrm N}.
\end{equation}
The corresponding lateral input and observation regions are denoted by
$\Sigma_{\mathrm D}$ and $\Sigma_{\mathrm N}$, respectively.
The spatial sets may overlap around $\partial M_{\mathrm{tan}}$.
Compactness
implies that for some $\delta_0>0$,
\begin{equation}\label{eq:strict-sign}
 \partial M\setminus \Gamma_{\mathrm D}\subset\{\partial_\nu\varphi\leq-\delta_0\},
 \qquad
 \partial M\setminus \Gamma_{\mathrm N}\subset\{\partial_\nu\varphi\geq\delta_0\}.
\end{equation}
This fixed strict-sign property is indispensable in a quantitative argument.
The smooth relative-boundary convention causes no loss for arbitrary open
neighborhoods in \eqref{eq:boundary-neighborhoods}: choose smaller smooth
neighborhoods still containing the respective closed faces. Inclusion of
their supported input spaces and restriction of their observed output
spaces are bounded, so their discrepancy is bounded by the discrepancy
on the originally given open sets.

The partial input--output map sends the initial state and the lateral
Dirichlet datum supported in $\Sigma_{\mathrm D}$ to the terminal state
and the Neumann trace on $\Sigma_{\mathrm N}$:
\[
 (u_0,f)\longmapsto (u(T),\partial_\nu u|_\V).
\]
We measure differences of these maps in the operator norm $\|\cdot\|_*$
defined in Section~\ref{sec:framework}. Its inputs have orders $-1$ in space
and $(-1/2,-1/4)$ on the lateral boundary; its outputs have orders $1$ and
$(1/2,1/4)$, respectively. This stronger output norm is available because
the difference solution has zero initial and Dirichlet data.

We next formulate the quantitative geometric hypothesis. Let
$\pi_0:\mathbb R\times M_0\to M_0$ be the canonical projection, and fix a
smooth compact domain $K\Subset M_0^{\mathrm{int}}$ such that
$\pi_0(M)\Subset K^\circ$. Work in a smooth extension $M_1$ of $M_0$ (the choice
$M_1=M_0$ is permitted). Let
$SM_1=\{(y,v)\in TM_1:|v|_{g_0}=1\}$ be the unit sphere bundle.
Its incoming boundary is
\[
 \partial_{\mathrm{in}}SM_1
 =\{(y,v):y\in\partial M_1,\ |v|=1,\ \langle v,\nu_1(y)\rangle<0\}.
\]
Choose an open set $\Gamma\Subset\partial_{\mathrm{in}}SM_1$ on which
the first exit time $\tau_+(y,v)$ is finite and smooth, and the exit is
non-tangential. Fix a real cutoff $\chi\in C_c^\infty(\Gamma)$,
$0\leq\chi\leq1$. Integration over $\Gamma$ always uses the incoming
Liouville measure
\[
 d\mu(y,v)=|\langle v,\nu_1(y)\rangle|\,
 dS_{\partial M_1}(y)\,dS_y(v).
\]
For each $z\in\supp\chi$, let
$\gamma_z:[0,\tau_+(z)]\to M_1$ be the corresponding maximal unit-speed
geodesic. Assume there are smoothly varying times $b(z)<e(z)$ such that
$\gamma_z([b(z),e(z)])\subset M_0$, its endpoints lie on $\partial M_0$
and are non-tangential, and
\begin{equation}\label{eq:segment-exhaustion}
 \{r\in[0,\tau_+(z)]:\gamma_z(r)\in K\}\subset (b(z),e(z)).
\end{equation}
Thus the chosen segment contains every visit of the full ray to $K$.
Set $\gamma(r)=\gamma_z(b(z)+r)$, $L_\gamma=e(z)-b(z)$, and
denote this parameterized compact collection by $\G$. It is used only
to index beams; its parameter measure is never chosen independently.
Let $\Gamma_{\mathrm{geo}}\subset\partial_{\mathrm{in}}SM_0$ denote the open
set of initial data for non-tangential geodesics with finite exit time;
we identify each geodesic with its initial data. The selected segments
lie in $\Gamma_{\mathrm{geo}}$. Define the attenuated geodesic ray transform and
its localization by $\chi$ as
\[
 I_af(z)=\int_0^{L_\gamma}e^{ar}f(\gamma(r))\dd r,
 \qquad I_{a,\chi}f=\chi I_af.
\]
All data norms below are $\|I_{a,\chi}f\|_{L^2(\Gamma,d\mu)}$.
For $\supp f\subset K$, the full-ray transform is exactly
$I_a^{(1)}f(z)=e^{ab(z)}I_af(z)$. This smooth positive factor and its
inverse have uniformly bounded derivatives on $\supp\chi$ for bounded
$a$. This observation connects the beam parameterization to the full-ray
normal operator.

\begin{definition}[Uniform geometric conditions]\label{def:beam-regular}
For a compact family \(\G\Subset\Gamma_{\mathrm{geo}}\), we impose the following
uniform geometric conditions. There are constants
\[
 \epsilon_*,\delta_{\mathrm{tube}},\delta_{\mathrm{ret}},
 \rho_*,L_*,J_*,\kappa_*>0
\]
with the following properties.  Every \(\gamma\in\G\) has
\(L_\gamma\leq L_*\), extends to
\([-\epsilon_*,L_\gamma+\epsilon_*]\), and admits a Fermi-coordinate cover
by at most \(J_*\) tubes of transverse radius
\(\delta_{\mathrm{tube}}\). In a fixed finite atlas of incoming parameters,
the Fermi charts are chosen smoothly in those parameters. For every integer
$k\geq0$, the metric coefficients, inverse metric coefficients, volume
Jacobians and their reciprocals, and coordinate transition maps and their
inverses have $C^k$ bounds $B_k$ uniform over the selected compact parameter
support and the extended intervals $[-\epsilon_*,L_\gamma+\epsilon_*]$.
These bounds include mixed derivatives in the spatial and incoming
parameters on the compact chart overlaps used in the construction.
The smooth extension and charts are fixed; no uniformity over changes of
the extension or metric is asserted. Each estimate below uses only finitely
many of the bounds $B_k$. At each endpoint,
\[
 |\langle\dot\gamma(0),\nu_0(\gamma(0))\rangle_{g_0}|
 \geq\kappa_*,
 \qquad
 |\langle\dot\gamma(L_\gamma),\nu_0(\gamma(L_\gamma))\rangle_{g_0}|
 \geq\kappa_*.
\]
Require $\rho_*$ to be below half the injectivity radius of the fixed
extension, and choose \(\delta_{\mathrm{ret}}\) below the injectivity radius of a fixed
extension of \(M_0\).  Whenever \(r,s\in[0,L_\gamma]\) satisfy
\[
 |r-s|\geq\rho_*,
 \qquad
 d_{g_0}(\gamma(r),\gamma(s))<\delta_{\mathrm{ret}},
\]
parallel transport \(\mathsf P_{s\to r}\) along the unique short minimizing
geodesic is defined and
\[
 |\dot\gamma(r)-\mathsf P_{s\to r}\dot\gamma(s)|_{g_0}
 \geq\kappa_*.
\]
\end{definition}

The last condition excludes returns with nearly coincident oriented
tangent directions. It supplies the uniform separation of the gradients
of the real phases in Lemma~\ref{lem:beam-cross}. Nearly opposite tangent
directions do not cause this phase difference to become stationary, so no
separation of unoriented tangent lines is required. For the fixed smooth
geometry, the bounds $B_k$ follow from smooth dependence of the Fermi charts
on the incoming data and compactness, after shrinking the charts.  The parameter separation \(\rho_*\)
distinguishes two distinct local branches from neighboring points on the same
local branch.  A compact family of geodesics meeting a fixed interior set in a
simple manifold satisfies these conditions: otherwise compactness in the incoming-data
topology would produce a limiting self-intersection with separated parameter
values.

\begin{assumption}[Stability of the attenuated geodesic ray transform]\label{ass:visibility}
The family \(\G\) satisfies Definition~\ref{def:beam-regular}, and there exist $a_0>0$, $\ell>1$,
and $C_{\mathrm{stab}}>0$ such that
\begin{equation}\label{eq:visibility}
 \|f\|_{H^{-\ell}(M_0)}
 \leq C_{\mathrm{stab}}\|I_{a,\chi}f\|_{L^2(\Gamma,d\mu)},
 \qquad |a|\leq a_0,\quad f\in C_c^\infty(M_0^{\mathrm{int}}),\quad\supp f\subset K.
\end{equation}
\end{assumption}

By approximation, \eqref{eq:visibility} also holds for
\(f\in H^1(M_0)\) supported in \(K\): use interior-supported smooth approximation in the fixed compact
domain $K$ (chosen with a margin around the projection of $M$), use the
continuous embedding \(H^1\hookrightarrow H^{-\ell}\), and use the boundedness
of \(I_{a,\chi}:H^1(M_0)\to L^2(\Gamma,d\mu)\), uniformly for \(|a|\leq a_0\).  We do not
assert an extension to arbitrary \(H^{-\ell}\) distributions, for which the
\(L^2(\Gamma,d\mu)\) ray transform need not be defined.  The geometric conditions in Definition~\ref{def:beam-regular} make the
Gaussian beam estimates uniform.  The assumption is a quantitative
replacement of ray-transform injectivity and is strictly stronger than the
hypothesis needed for uniqueness.  It is verified on simple transversal
manifolds and on regular ray families; see Theorem~\ref{cor:simple}
and Corollary~\ref{cor:regular}.

Fix $m>0$ and define the admissible class of potentials
\begin{equation}\label{eq:potential-class}
 \Q_m=\{q\in W^{1,\infty}(Q;\mathbb R):
              \|q\|_{W^{1,\infty}(Q)}\leq m\}.
\end{equation}
We write \(H^{-1}(Q)=(H_0^1(Q))'\), with the equivalent norm obtained by
restriction from a fixed product neighborhood of \(\overline Q\).
For $s>0$, define the continuous increasing stability modulus
\begin{equation}\label{eq:modulus}
 \Psi_s(r)=\begin{cases}
 0,&r=0,\\
 r+(\log|\log r|)^{-s},&0<r<e^{-e},\\
 r+1,&r\geq e^{-e}.
 \end{cases}
\end{equation}

\begin{theorem}[Simple transversal manifold]\label{cor:simple}
Assume \eqref{eq:cta} and \eqref{eq:boundary-neighborhoods}, and suppose
that $(M_0,g_0)$ is simple. For every $0<s<1/8$ there is $C_s>0$,
depending only on the fixed geometry, $T,m,\Gamma_{\mathrm D},\Gamma_{\mathrm N},s$, such that
\[
 \|q_1-q_2\|_{H^{-1}(Q)}\leq
 C_s\Psi_s\!\left(\|\Lambda^{\U,\V}_{c,g,q_1}
 -\Lambda^{\U,\V}_{c,g,q_2}\|_*\right)
\]
for all $q_1,q_2\in\Q_m$ agreeing on $\Sigma$. No ray transform stability assumption is required in addition to simplicity.
\end{theorem}

\begin{theorem}[Conditional CTA stability]\label{thm:main}
Assume \eqref{eq:cta}, \eqref{eq:boundary-neighborhoods}, and
Assumption~\ref{ass:visibility}.  Let $0<s<1/[4(\ell+1)]$.  There is a constant
$C>0$, depending only on the a priori geometry, $T$, $m$, $\Gamma_{\mathrm D}$, $\Gamma_{\mathrm N}$,
$a_0,\ell,C_{\rm stab},s$, the cutoff $\chi$, the fixed incoming density,
and the uniform geometric constants in Definition~\ref{def:beam-regular},
such that for any $q_1,q_2\in\Q_m$ satisfying
\begin{equation}\label{eq:boundary-agreement}
 q_1=q_2\quad\text{on }\Sigma,
\end{equation}
one has
\begin{equation}\label{eq:main-stability}
 \|q_1-q_2\|_{H^{-1}(Q)}
 \leq C\Psi_s\!\left(
 \|\Lambda^{\U,\V}_{c,g,q_1}
       -\Lambda^{\U,\V}_{c,g,q_2}\|_*
 \right).
\end{equation}
\end{theorem}

\begin{corollary}[Regular families of geodesics]
\label{cor:regular}
Suppose that $M_0\Subset M_1^{\mathrm{int}}$. Let
$\Gamma\Subset\partial_{\mathrm{in}}SM_1$ be open, fix
$\chi\in C_c^\infty(\Gamma)$, and assume that the first-exit time is
finite and smooth on a neighborhood of $\supp\chi$. Use the incoming
density $d\mu$ specified above. Assume the following conditions.
\begin{enumerate}[label=\textup{(\roman*)}]
\item There are smooth entry and exit times $b(z)<e(z)$ such that
the selected interval is a connected non-tangential segment in $M_0$.
For some fixed open $O$, $K\Subset O\Subset M_0^{\mathrm{int}}$,
every visit of the full ray to $\overline O$ lies in $(b(z),e(z))$.
The selected segments satisfy Definition~\ref{def:beam-regular}. Other components
of a ray in $M_0$, if present, are disjoint from $\overline O$.
\item On a neighborhood of $\supp\chi$, the full rays are embedded
and have no conjugate points between any two of their points.
For every $(x,\zeta)\in T^*\overline O\setminus0$, a ray with
$\chi(z)\ne0$ passes through $x$ with $\zeta(\dot\gamma_z)=0$.
Here absence of conjugate points means that the differential of
the exponential map in the initial direction and travel time has
full rank away from zero travel time. Thus, over $\overline O$, the
selected rays form a regular family of curves in the sense of
\cite[Definition~1]{FrigyikStefanovUhlmann2008}. The incoming boundary
$\partial_{\mathrm{in}}SM_1$ provides the required transverse parametrization.
\item $I_{0,\chi}$ is injective on $L^2(M_1)$ functions supported in $K$.
\end{enumerate}

Then \eqref{eq:main-stability} holds for every \(0<s<1/8\), without imposing
Assumption~\ref{ass:visibility} separately.  The manifold \(M_0\) itself may
be non-simple; geodesics outside \(\Gamma\) may have conjugate points or be
trapped.
\end{corollary}

\subsection*{Relation to the companion uniqueness paper}

The related qualitative inverse problem on general CTA manifolds is
treated in the public preprint \cite{ZhengDuanJing2026}, which assumes
injectivity of the small-attenuation ray transforms and establishes
uniqueness. It does not establish a stability modulus.
Assumption~\ref{ass:visibility} is strictly stronger: it requires
a uniform stability estimate on a compact family of rays.  Consequently,
Theorem~\ref{thm:main} does not supersede the general CTA uniqueness theorem;
it gives a stronger conclusion on the class satisfying Assumption~\ref{ass:visibility}.
Corollary~\ref{cor:regular} gives a concrete genuinely non-simple subclass:
only the selected regular family must resolve \(T^*K\), while the unused
geodesics are unrestricted.

Several ingredients below have qualitative counterparts in
\cite{ZhengDuanJing2026}: the graph-space
formulation of the partial input--output map, the forward and transpose
Carleman estimates, the construction of two solutions with lateral traces
supported in \(\U\) and \(\V\), and the limiting Gaussian beam identity.  The
new points needed for stability are the improved range of the difference of
the input--output maps, an \(O(h^{1/2})\) concentration estimate uniform in the geodesic,
bounds for the exact solutions uniform in the frequency parameters, a quantitative inversion of
the attenuated ray transform, and a Hilbert-valued continuation estimate with
an explicit harmonic-measure exponent.  The shared qualitative framework is included because stability requires
uniform estimates for each step of the argument. Weak convergence of beam measures
is replaced by a uniform $O(h^{1/2})$ error; exact solutions for fixed
parameters are replaced by bounds uniform in $(\gamma,\sigma,\tau)$;
injectivity is replaced by stable inversion for small attenuation; and
vanishing of the Fourier transform on a fixed interval is replaced by
Hilbert-valued continuation to an expanding interval. These estimates,
together with interpolation, yield the explicit double-logarithmic modulus.
The quantitative arguments needed below are proved in the present manuscript.

The proof combines two quantitative mechanisms.  The parabolic Carleman
weight produces a factor $e^{C/h^2}$ in the measurement term.  Optimizing this
factor yields a first logarithm.  Gaussian beams recover the attenuated ray
transform only for a fixed interval of Euclidean frequencies.  An
analytic continuation argument based on subharmonicity enlarges that interval and
produces the second logarithm.  This is analogous to the extra logarithm in
the admissible elliptic problem of Caro and Salo \cite{CaroSalo2014}, but the
present parabolic construction is based on geometric optics solutions with prescribed lateral support
and an input--output map containing both temporal endpoint traces.

Choulli and Kian \cite{ChoulliKian2018} proved double-logarithmic stability for
a time-dependent parabolic potential in a Euclidean domain from a partial
Dirichlet-to-Neumann map, without a terminal observation. Our result permits
CTA geometry and restricts the lateral input support, but includes the
terminal state in the output and the initial state in the input. The two
measurement settings are therefore not directly comparable. In the present
setting no temporal cutoff is required in the exact solutions, which allows
the $O(h^{1/2})$ bulk remainder. The additional geometric task is to control
Gaussian beam concentration uniformly and stably invert the resulting
attenuated transform.

Logarithmic moduli are characteristic of coefficient recovery from boundary
measurements, beginning with the elliptic stability mechanism of
Alessandrini \cite{Alessandrini1988}.  For time-dependent parabolic
coefficients, quantitative continuation and Carleman estimates were developed
in \cite{ChoulliKian2013,ChoulliKian2018}; recent partial-data results for
reaction-diffusion and heat equations further illustrate the role of
geometrically restricted observations \cite{FeizmohammadiKianUhlmann2024}.
The second logarithm in the present result reflects the additional analytic
continuation from a fixed interval of Euclidean frequencies.

Recent work also considers other parabolic observation settings.
Alberola, Aroua, and Caro \cite{AlberolaArouaCaro2026} establish uniqueness
for a time-dependent heat potential on $\mathbb R^n$ from the
initial-to-final map under super-exponential spatial decay assumptions.
Their result uses no lateral boundary data and does not provide a stability
modulus. Klibanov \cite{Klibanov2026} proves uniqueness and stability for
the simultaneous recovery of a space--time-dependent coefficient and an
unknown initial condition in a general linear parabolic equation.
These settings differ from the partial input--output operator and CTA
geometry considered here.

Earlier parabolic coefficient-identification results include the
boundary response approach of Avdonin and Seidman
\cite{AvdoninSeidman1995} and the boundary spectral approach to heat
equations of Canuto and Kavian \cite{CanutoKavian2001}.
Completeness of products of solutions provides another general route
to uniqueness \cite{Isakov1991}. For restricted lateral observations,
Fan and Duan \cite{FanDuan2021} treat a parabolic potential, while
Bellassoued and Ben Fraj \cite{BellassouedFraj2021} obtain quantitative
results for convection--diffusion coefficients. Yamamoto's survey
\cite{Yamamoto2009} describes the parabolic Carleman methods underlying
many coefficient and source estimates. A recent JMAA contribution by
Nguyen Van Thang and Nguyen Van Duc \cite{NguyenNguyen2025} concerns
stability for a time-dependent coefficient in a one-dimensional
parabolic equation. Its coefficient class and observation setting
differ from the space--time potential and manifold geometry studied here.

On admissible manifolds, Mishra, Purohit, and Vashisth
\cite{MishraPurohitVashisth2025} established partial-data uniqueness for
time-dependent first- and zeroth-order perturbations, modulo the natural
gauge invariance. Their result is qualitative and does not provide a
stability modulus for the CTA input--output map considered here.
In Euclidean domains, Buisson \cite{Buisson2026} establishes partial-data
uniqueness for space--time-dependent convection and potential coefficients
from a partial Dirichlet-to-Neumann map. These results concern simultaneous
coefficient recovery, whereas the present work quantifies recovery of the
potential with the principal part fixed.

A closely related CTA result is due to Liu and Purohit
\cite{LiuPurohit2026}, who recover time-dependent convection and density
coefficients from a partial input--output operator, up to the natural gauge
invariance, under an injectivity assumption for attenuated ray transforms
of functions and one-forms on the transversal manifold. Their work shares
the CTA geometry and the use of Gaussian beams and Carleman estimates,
but establishes qualitative uniqueness rather than a stability modulus.
The present paper treats the scalar potential with the first-order term
fixed at zero and imposes a uniform quantitative stability hypothesis for
the scalar attenuated ray transform, together with uniform geometric
conditions on the selected rays. Under these stronger quantitative
hypotheses, uniform beam concentration and exact-solution estimates,
stable inversion, and Hilbert-valued analytic continuation yield the
double-logarithmic estimate. Thus our result quantifies potential recovery
on the stated class; it does not replace their simultaneous coefficient
uniqueness theorem.

On manifolds, the relation between geometric optics and ray transforms
also appears in the admissible advection--diffusion problem
\cite{KrupchykUhlmann2018} and in the time-dependent hyperbolic inverse
problems of Liu, Saksala, and Yan
\cite{LiuSaksalaYan2024,LiuSaksalaYan2025}. The latter works provide
qualitative geometric precedents; the present conclusion requires
uniform quantitative bounds in addition to injectivity. For this
reason the hypotheses on the selected geodesics include both uniform
geometric conditions and a stability estimate for the ray transform, and the regular-family corollary explicitly includes
injectivity instead of inferring it from conormal coverage alone.

The paper is organized as follows.  Section~\ref{sec:framework} records the
functional framework and conformal reduction.  Section~\ref{sec:beams}
establishes quantitative Gaussian beam concentration.  The Carleman
solvability and exact geometric optics solutions are constructed in
Section~\ref{sec:cgo}.  Section~\ref{sec:ray} derives the quantitative ray
transform estimate.  Section~\ref{sec:continuation} proves the continuation
lemma and completes the stability argument.  The results for simple manifolds and regular families are proved in
Section~\ref{sec:simple-proof}.

\section{Functional framework and conformal reduction}\label{sec:framework}

For $r,s\geq0$, define the anisotropic Sobolev space
\[
 H^{r,s}(\Sigma)=L^2(0,T;H^r(\partial M))
 \cap H^s(0,T;L^2(\partial M)),
\]
with the Hilbert norm
\[
 \|F\|_{H^{r,s}(\Sigma)}^2
 =\|F\|_{L^2(0,T;H^r(\partial M))}^2
  +\|F\|_{H^s(0,T;L^2(\partial M))}^2.
\]
For $r,s>0$, let $H_0^{r,s}(\Sigma)$ be the closure of
$C_c^\infty(\Sigma)$ in $H^{r,s}(\Sigma)$, and define
\[
 H^{-r,-s}(\Sigma)=\big(H_0^{r,s}(\Sigma)\big)'.
\]
If $W\subset\Sigma$ is open, $H^{\alpha,\beta}(W)$ denotes the space of
restrictions to $W$ of functions or distributions in
$H^{\alpha,\beta}(\Sigma)$.  It is equipped with the quotient norm
\[
 \|G\|_{H^{\alpha,\beta}(W)}
 =\inf\big\{\|F\|_{H^{\alpha,\beta}(\Sigma)}:
              F\in H^{\alpha,\beta}(\Sigma),\ F|_W=G\big\}.
\]
For $W=(0,T)\times W_0$ with smooth relative boundary, define
\begin{equation}\label{eq:supported-space}
 \widetilde H^{-1/2,-1/4}(W)
 =\{f\in H^{-1/2,-1/4}(\Sigma):\supp f\subset\overline W\},
\end{equation}
with the norm inherited from the global negative-order space.
This is the supported Sobolev space; closure and support are relative to $\Sigma$.

\begin{lemma}[Duality of restriction and supported trace spaces]\label{lem:local-duality}
There is an isometric identification, for the quotient and dual norms,
\[
 (H^{1/2,1/4}(W))'=\widetilde H^{-1/2,-1/4}(W).
\]
If $g=G|_W$ and $f$ belongs to the supported space, the local pairing is
$\langle g,f\rangle_W=\langle G,f\rangle_\Sigma$, independently of $G$.
In particular, $|\langle g,f\rangle_W|\leq
\|g\|_{H^{1/2,1/4}(W)}\|f\|_{\widetilde H^{-1/2,-1/4}(W)}$.
\end{lemma}
\begin{proof}
Put $E=H^{1/2,1/4}(\Sigma)$ and let $R:E\to E|_W$ be restriction.
Since $1/4<1/2$, time cutoffs show $H_0^{1/2,1/4}(\Sigma)=E$.
Quotient duality identifies $(E/\ker R)'$ with the annihilator
$(\ker R)^\perp\subset E'$, without invoking zero extension of $g$.
An annihilator vanishes on every test function supported outside
$\overline W$, so its distributional support lies in $\overline W$.
Conversely, a function $G\in\ker R$ can be approximated in $E$ by
smooth functions compactly supported in $\Sigma\setminus\overline W$.
Here is the relevant density at the spatial endpoint $1/2$: flatten
$\partial W_0$, localize, and translate the zero-side function a positive
distance into the complement before mollifying. Translation is strongly
continuous both in $L^2_tH^{1/2}_y$ and in $H^{1/4}_tL^2_y$; smooth
coordinate changes and a finite partition of unity preserve these norms.
Time cutoffs and mollification finish the approximation. Every supported
$f\in E'$ consequently annihilates $\ker R$. The transpose of the
quotient map is isometric, which proves both the identification and the
pairing bound. No bounded zero-extension map on the positive restriction
space at spatial order $1/2$ is asserted.
\end{proof}

We define the space of admissible lateral Dirichlet inputs by
\[
 \D_\U=\widetilde H^{-1/2,-1/4}(\U)
 :=\big\{f\in H^{-1/2,-1/4}(\Sigma):
              \supp f\subset\overline\U\big\}.
\]
Here the support and the closure of \(\U\) are taken relative to
\(\Sigma\).  We give \(\D_\U\) the norm inherited from
\(H^{-1/2,-1/4}(\Sigma)\).  Since \(\D_\U\) is a closed subspace, it is a
Hilbert space.  The full input space is
\[
 \X_\U=H^{-1}(M)\times\D_\U.
\]
We equip it with the product norm
\[
 \|(u_0,f)\|_{\X_\U}^2
 =\|u_0\|_{H^{-1}(M)}^2
  +\|f\|_{H^{-1/2,-1/4}(\Sigma)}^2.
\]
For $(u_0,f)\in\X_\U$, the transposition solution of
\eqref{eq:ibvp} belongs to $L^2(Q)$ and has generalized terminal and Neumann
traces.  The partial input--output map is
\begin{equation}\label{eq:io-map}
 \Lambda^{\U,\V}_{c,g,q}:\X_\U
 \longrightarrow H^{-1}(M)\times H^{-3/2,-3/4}(\V),
 \qquad
 (u_0,f)\longmapsto\big(u(T),\partial_\nu u|_\V\big).
\end{equation}
If $u_j$ is the solution corresponding to $q_j$ with the same input and
$w=u_1-u_2$, then $w$ has zero initial and lateral Dirichlet traces and
satisfies
\[
 P_{q_1}w=(q_2-q_1)u_2\in L^2(Q).
\]
By maximal $L^2$ parabolic regularity and the endpoint and normal trace
theorems,
\[
 w\in H^{2,1}(Q),\qquad
 w(T)\in H_0^1(M),\qquad
 \partial_\nu w\in H^{1/2,1/4}(\Sigma).
\]
Thus the difference of the maps has the improved range
\begin{equation}\label{eq:improved-range}
 \Y_\V=H_0^1(M)\times H^{1/2,1/4}(\V),
 \qquad
 \Lambda^{\U,\V}_{c,g,q_1}-\Lambda^{\U,\V}_{c,g,q_2}
 \in\mathcal B(\X_\U,\Y_\V).
\end{equation}
The output space is equipped with
\[
 \|(v,G)\|_{\Y_\V}^2
 =\|v\|_{H_0^1(M)}^2+\|G\|_{H^{1/2,1/4}(\V)}^2.
\]
We use the following operator norm for the measurement discrepancy:
\begin{equation}\label{eq:measurement-norm}
 \begin{split}
 &\|\Lambda^{\U,\V}_{c,g,q_1}
       -\Lambda^{\U,\V}_{c,g,q_2}\|_*\\
 &\quad:=
 \sup_{\substack{(u_0,f)\in\X_\U\\(u_0,f)\neq(0,0)}}
 \frac{\|\big(\Lambda^{\U,\V}_{c,g,q_1}
                 -\Lambda^{\U,\V}_{c,g,q_2}\big)(u_0,f)\|_{\Y_\V}}
      {\|(u_0,f)\|_{\X_\U}}.
 \end{split}
\end{equation}
Equivalently, if \(u_j\) is the solution for \(q_j\) with the common input
\((u_0,f)\) and \(w=u_1-u_2\), then the numerator in
\eqref{eq:measurement-norm} is
\[
 \left(\|w(T)\|_{H_0^1(M)}^2
       +\|\partial_\nu w|_\V\|_{H^{1/2,1/4}(\V)}^2\right)^{1/2}.
\]
Thus \(\|\cdot\|_*\) is the norm of the \emph{difference} map in the improved
space \(\mathcal B(\X_\U,\Y_\V)\), not the norm of either individual map in
its natural weaker range in \eqref{eq:io-map}. This gain follows from
cancellation of the common initial and Dirichlet data: the difference solves
a problem with homogeneous data and an $L^2$ source. The theorem measures
errors in this stronger operator norm. It does not identify that norm
with the weaker norm of arbitrary perturbations of a single forward map.

For later reference, we use the anisotropic Sobolev space
\[
 H^{2,1}(Q)=L^2(0,T;H^2(M))\cap H^1(0,T;L^2(M))
\]
with its natural intersection norm.  All duality pairings below are complex
bilinear when the formal transpose is used; this convention agrees with the
integral identity in Section~\ref{sec:ray}.

\begin{lemma}[Lifting of compatible boundary and temporal traces]\label{lem:joint-lifting}
Given $g\in H_0^{3/2,3/4}(\Sigma)$,
$\eta\in H^{1/2,1/4}(\Sigma)$ and $\psi_0,\psi_T\in H_0^1(M)$,
there is a bounded linear lifting $Z\in H^{2,1}(Q)$ satisfying
\[
 Z|_\Sigma=g,\qquad \partial_\nu Z|_\Sigma=\eta,\qquad
 Z(0)=\psi_0,\qquad Z(T)=\psi_T.
\]
Here the continuous endpoint maps are
\[
 r_\kappa:H^{3/2,3/4}(\Sigma)\longrightarrow H^{1/2}(\partial M),
 \qquad \kappa\in\{0,T\},
\]
and $H_0^{3/2,3/4}(\Sigma)=\ker r_0\cap\ker r_T$, with the
inherited intersection norm. In particular $g(0)=g(T)=0$ means zero
in $H^{1/2}(\partial M)$, and agrees with the traces of
$\psi_0,\psi_T\in H_0^1(M)$. Thus the compatibility conditions at the space-time corners hold;
the normal data, whose time exponent is $1/4$, have no endpoint trace.
\end{lemma}
\begin{proof}
Trace continuity is the parabolic Hilbert trace theory of
\cite[Chapter~4]{LionsMagenes}. More explicitly, the anisotropic
trace theorem \cite[Theorem~4.6.2]{Amann2009}, applied with anisotropy
$(2,1,\ldots,1)$, $p=2$, $s=3/2$, and trace order $j=0$,
gives the endpoint trace space of order $s-2/p=1/2$; its target
$B^{1/2}_{2,2}$ equals $H^{1/2}$. The intersection-space realization
is \cite[Theorem~3.7.3]{Amann2009}. Localization in boundary charts
and time reflection give the result on the finite cylinder.
We give the right inverses used here explicitly. On the closed manifold $\partial M$ let
$-\Delta_{\partial M}e_j=\lambda_j e_j$, and put $\omega_k=k\pi/T$.
On the basis $e_j(y)\sin(\omega_k t)$ define the positive multiplier
$A$ by $A_{jk}=(1+\lambda_j+\omega_k)^{1/2}$.
This multiplier is intrinsic: it is the joint spectral calculus of
$1-\Delta_{\partial M}$ and $(-\partial_{t,D}^2)^{1/2}$, which act
on separate variables. Changing orthonormal bases within eigenspaces
does not change $A$. The inward normal collar is defined by the fixed
metric; changing boundary charts merely rewrites the same construction.
The associated spectral norms are equivalent to the chart-defined
Sobolev norms on the compact smooth boundary.
Odd reflection in time gives the norm equivalences
\[
 \|A^{1/2}\eta\|_{L^2}\asymp\|\eta\|_{H^{1/2,1/4}},\qquad
 \|A^{3/2}g\|_{L^2}\asymp\|g\|_{H^{3/2,3/4}}.
\]
The first needs no endpoint condition since $1/4<1/2$; the second
uses the zero endpoint traces since $1/2<3/4<3/2$.
In an inward normal collar of width $\rho_0$ take a cutoff $\theta$
equal to one near $\rho=0$ and zero near $\rho_0$. Set
\[
 L_N\eta=-\theta(\rho)\rho e^{-\rho A}\eta,\qquad
 L_Dg=\theta(\rho)(1+\rho A)e^{-\rho A}g.
\]
Since $\partial_\nu=-\partial_\rho$, these have lateral traces
$(0,\eta)$ and $(g,0)$, respectively. Both temporal traces vanish.
Indeed finite sine sums have this property, and convergence in
$H^{2,1}$ implies convergence of the temporal traces in $H^1(M)$.
For each multiplier value $a\geq1$, integration in $\rho$ gives
\[
 \int_0^\infty a^4\rho^2e^{-2a\rho}\,d\rho\leq Ca,
 \qquad
 \int_0^\infty a^4(1+a\rho)^2e^{-2a\rho}\,d\rho\leq Ca^3.
\]
The same bounds hold for two normal derivatives and for one time
derivative, since $\omega_k\leq a^2$. The same estimates hold for
all tangential and mixed spatial derivatives of total order at most two:
each tangential derivative is controlled by one power of $a$ in the
boundary elliptic norm, and each normal derivative contributes at most
one further power. Parseval and elliptic norm equivalence on $\partial M$
therefore prove the asserted boundedness.
Derivatives of $\theta$ and the smooth collar metric add only lower
order terms with the same bounds. Extension by zero past the collar
is smooth at its inner edge.

For the temporal data use $A_D=1-\Delta_D$ and disjoint smooth
cutoffs $\theta_0,\theta_T$ equal to one near the respective endpoints:
\[
 E\psi=\theta_0(t)e^{-tA_D}\psi_0+
 \theta_T(t)e^{-(T-t)A_D}\psi_T.
\]
The spectral identity $D(A_D^{1/2})=H_0^1(M)$ and
$\int_0^\infty a^2e^{-2ta}\,dt=a/2$ show that
$E\psi\in H^{2,1}(Q)$ with the required temporal traces and zero
Dirichlet trace. Its normal trace $\eta_E$ is in $H^{1/2,1/4}$.
Thus $L_Dg+E\psi+L_N(\eta-\eta_E)$ is the desired lifting.
For completeness, continuity of $r_0,r_T$ shows that the closure of
$C_c^\infty(\Sigma)$ is contained in their common kernel. Conversely,
zero endpoint traces allow zero extension in
$H^{3/4}(\mathbb R;L^2(\partial M))$, by the fractional Hardy
inequality at each endpoint; zero extension also preserves the
$L^2(\mathbb R;H^{3/2}(\partial M))$ norm. Dilate this extension
slightly about $T/2$ to place its support strictly inside $(0,T)$.
Strong continuity of dilation in both norms, followed by time
mollification and boundary spectral smoothing, gives approximation by
$C_c^\infty(\Sigma)$ in the intersection norm. This proves the stated
kernel characterization. No critical temporal exponent occurs here:
$3/4$ is strictly between $1/2$ and $3/2$, whereas the normal data
have temporal order $1/4<1/2$.
\end{proof}

With the spaces defined at the beginning of this section, the liftings
above will be used in the following concrete form. If
\(\eta\in H^{1/2,1/4}(\Sigma)\), then there is
\(Z_\eta\in H^{2,1}(Q)\) satisfying
\begin{equation}\label{eq:normal-extension}
 Z_\eta|_\Sigma=0,\qquad
 \partial_\nu Z_\eta|_\Sigma=\eta,\qquad
 Z_\eta(0)=Z_\eta(T)=0,
 \qquad
 \|Z_\eta\|_{H^{2,1}(Q)}\leq C\|\eta\|_{H^{1/2,1/4}(\Sigma)}.
\end{equation}
If \(\zeta\in H_0^{3/2,3/4}(\Sigma)\), there is an analogous
\(Z_\zeta\) with
\begin{equation}\label{eq:dirichlet-extension}
 Z_\zeta|_\Sigma=\zeta,\qquad
 \partial_\nu Z_\zeta|_\Sigma=0,\qquad
 Z_\zeta(0)=Z_\zeta(T)=0,
 \qquad
 \|Z_\zeta\|_{H^{2,1}(Q)}\leq C\|\zeta\|_{H^{3/2,3/4}(\Sigma)}.
\end{equation}
We also use the corresponding extension of compatible temporal traces.  These
right inverses may be constructed in boundary normal coordinates and patched
by a finite partition of unity.  In particular, their norms depend only on
the fixed geometry and \(T\); see, for example, the parabolic trace framework
in \cite[Chapter~4]{LionsMagenes}; the required right inverses were proved above.

We first work with the product metric $\widetilde g=\dd x_1^2\oplus g_0$.
Set
\[
 P_q=\partial_t-\Delta_{\widetilde g}+q,
 \qquad P_q^t=-\partial_t-\Delta_{\widetilde g}+q,
\]
where the superscript $t$ denotes the formal transpose for the complex
bilinear Green identity.  Define the maximal parabolic graph spaces
\[
 \Hc_+(Q)=\{u\in L^2(Q):(\partial_t-\Delta_{\widetilde g})u\in L^2(Q)\},
 \quad
 \Hc_-(Q)=\{u\in L^2(Q):(-\partial_t-\Delta_{\widetilde g})u\in L^2(Q)\}.
\]
The lateral Dirichlet and Neumann traces and both temporal traces extend as
bounded maps
\begin{align}\label{eq:graph-traces}
 &\tau_0:\Hc_\pm(Q)\to H^{-1/2,-1/4}(\Sigma),
 &&\tau_1:\Hc_\pm(Q)\to H^{-3/2,-3/4}(\Sigma),\\
 &r_0,r_T:\Hc_\pm(Q)\to H^{-1}(M).\notag
\end{align}

\subsection{Generalized traces and transposition solutions}

We include the proof because the stronger range used for the difference of
two measurements depends on the precise trace spaces.  For smooth functions
\(u,Z\), Green's formula reads
\begin{align}\label{eq:general-green}
 &((\partial_t-\Delta_{\widetilde g})u,Z)_{L^2(Q)}
 -(u,(-\partial_t-\Delta_{\widetilde g})Z)_{L^2(Q)}\notag\\
 &\quad=\langle\tau_0u,\partial_\nu Z\rangle_\Sigma
 -\langle\tau_1u,Z|_\Sigma\rangle_\Sigma
 +\langle r_Tu,Z(T)\rangle_M-\langle r_0u,Z(0)\rangle_M.
\end{align}
Here the first lateral pairing is between
\(H^{-1/2,-1/4}\) and \(H^{1/2,1/4}\), the second between
\(H^{-3/2,-3/4}\) and \(H_0^{3/2,3/4}\), and the temporal
pairings are between \(H^{-1}(M)\) and \(H_0^1(M)\).

\begin{lemma}[Density in the maximal graph spaces]\label{lem:graph-density}
$C^\infty(\overline Q)$ is dense in $\Hc_+(Q)$ and $\Hc_-(Q)$
for the graph norm. No boundary condition is imposed on the approximants.
\end{lemma}
\begin{proof}
It suffices to prove the forward case. Suppose that a continuous graph
functional vanishes on all smooth functions. By Hahn--Banach it has
the form $\mathcal L(u)=(a,u)+(b,P_0u)$ with $a,b\in L^2(Q)$.
Embed $M$ in a closed smooth manifold with an extended metric, and
extend $a,b$ by zero in space and time. Testing against restrictions
of globally smooth functions gives
$P_0^t b^0=-a^0$ on the full product in distributions.
Global $L^2$ parabolic regularity gives
$b^0\in H^1(\mathbb R;L^2)\cap L^2(\mathbb R;H^2)$.
For clarity this regularity follows by expansion in eigenfunctions of
the extended Laplacian and Fourier transformation in time: the symbol
$-i\tau+\lambda_k$ controls both $|\tau|$ and $\lambda_k$, with the
zero eigenvalue controlled by $\|b^0\|_{L^2}$.
Since $b^0$ is supported in $[0,T]\times M$, its two temporal traces,
lateral Dirichlet trace, and lateral normal trace vanish. The
zero-trace density theorem for $H^{2,1}$ (equivalently, translation
into the cylinder followed by smoothing in flattened charts) supplies
$b_j\in C_c^\infty(Q)$ with $b_j\to b$ in $H^{2,1}(Q)$.
For any $u\in\Hc_+(Q)$, the distributional equation defining $P_0u$
therefore yields
\[
 (b,P_0u)=\lim_j(b_j,P_0u)
 =\lim_j(P_0^tb_j,u)=(P_0^tb,u)=-(a,u).
\]
Hence every such annihilating functional is zero on the entire graph
space, which proves density. Time reversal proves the other case.
This argument uses no generalized trace of $u$ in establishing density.
\end{proof}

\begin{proposition}[Generalized parabolic traces]\label{prop:graph-traces}
The four classical traces extend uniquely to the bounded maps in
\eqref{eq:graph-traces}.  More precisely,
\begin{align}\label{eq:graph-trace-bound}
 &\|\tau_0u\|_{H^{-1/2,-1/4}(\Sigma)}
 +\|\tau_1u\|_{H^{-3/2,-3/4}(\Sigma)}\notag\\
 &\qquad+\|r_0u\|_{H^{-1}(M)}+\|r_Tu\|_{H^{-1}(M)}
 \leq C\|u\|_{\Hc_\pm(Q)}.
\end{align}
Identity \eqref{eq:general-green} remains valid whenever its pairings are
defined by the displayed spaces.
\end{proposition}

\begin{proof}
We treat \(\Hc_+(Q)\); time reversal gives the other case.  Let first
\(u\in C^\infty(\overline Q)\).  Insert the extension
\(Z_\eta\) from \eqref{eq:normal-extension} into
\eqref{eq:general-green}.  All boundary terms except the one containing
\(\tau_0u\) vanish, and therefore
\[
 |\langle\tau_0u,\eta\rangle_\Sigma|
 \leq C\big(\|u\|_{L^2(Q)}
       +\|(\partial_t-\Delta_{\widetilde g})u\|_{L^2(Q)}\big)
       \|\eta\|_{H^{1/2,1/4}(\Sigma)}.
\]
This defines \(\tau_0u\) in \(H^{-1/2,-1/4}(\Sigma)\).  Using
\(Z_\zeta\) from \eqref{eq:dirichlet-extension} in the same way gives
\[
 |\langle\tau_1u,\zeta\rangle_\Sigma|
 \leq C\|u\|_{\Hc_+(Q)}
       \|\zeta\|_{H^{3/2,3/4}(\Sigma)}.
\]

To obtain the endpoint traces, fix \(\psi\in H_0^1(M)\) and choose
\(Z\in H^{2,1}(Q)\) such that
\[
 Z|_\Sigma=0,\qquad Z(0)=0,\qquad Z(T)=\psi,\qquad
 \|Z\|_{H^{2,1}(Q)}\leq C\|\psi\|_{H_0^1(M)}.
\]
The already constructed \(\tau_0u\) and \eqref{eq:general-green} give
\begin{align*}
 |\langle r_Tu,\psi\rangle_M|
 &\leq C\big(\|u\|_{L^2(Q)}
 +\|(\partial_t-\Delta_{\widetilde g})u\|_{L^2(Q)}
 +\|\tau_0u\|_{H^{-1/2,-1/4}}\big)
 \|\psi\|_{H_0^1(M)}.
\end{align*}
The initial trace is identical.  Finally, Lemma~\ref{lem:graph-density} gives density of
$C^\infty(\overline Q)$ in the maximal parabolic graph spaces.
Passing to the limit proves the extensions, their uniqueness, and
\eqref{eq:graph-trace-bound}; the same density argument extends
\eqref{eq:general-green}.
\end{proof}

The trace-extension theorem also gives the density statement used later in
the Carleman solvability argument.

\begin{lemma}[Density of normal traces]\label{lem:normal-density}
Let \(W\subset\Sigma\) be relatively open and put
\[
 \mathcal T=\{z\in C^\infty(\overline Q):
 z|_\Sigma=0,\ z(0)=z(T)=0\}.
\]
Then \(\{\partial_\nu z|_W:z\in\mathcal T\}\) is dense in
\(L^2(W)\).  The traces on two disjoint relatively open portions of
\(\Sigma\) can be prescribed independently up to density.
\end{lemma}

\begin{proof}
For \(\eta\in C_c^\infty(W)\), extend \(\eta\) by zero to
\(\Sigma\).  In boundary normal coordinates \((y,\rho)\), take
\(z(t,y,\rho)=\rho\chi(\rho)\eta(t,y)\), with the sign of \(\rho\)
adjusted to the outward normal and \(\chi=1\) near zero.  A finite partition
of unity gives \(z\in\mathcal T\) with
\(\partial_\nu z|_\Sigma=\eta\).  Since
\(C_c^\infty(W)\) is dense in \(L^2(W)\), the first assertion
follows.  Applying the construction to the sum of two zero extensions proves
independent density on disjoint portions.
\end{proof}

\begin{proposition}[Direct problem and measurement map]
\label{prop:transposition}
Let \(q\in L^\infty(Q;\mathbb R)\).  For every
\[
 F\in L^2(Q),\qquad
 f\in H^{-1/2,-1/4}(\Sigma),\qquad
 u_0\in H^{-1}(M),
\]
there is a unique transposition solution \(u\in\Hc_+(Q)\) satisfying
\begin{equation}\label{eq:transposition-problem}
 P_qu=F\quad\text{in }Q,\qquad
 \tau_0u=f,\qquad r_0u=u_0.
\end{equation}
Moreover,
\begin{align}\label{eq:transposition-bound}
 &\|u\|_{\Hc_+(Q)}+\|r_Tu\|_{H^{-1}(M)}
 +\|\tau_1u\|_{H^{-3/2,-3/4}(\Sigma)}\notag\\
 &\qquad\leq C\big(\|F\|_{L^2(Q)}
 +\|f\|_{H^{-1/2,-1/4}(\Sigma)}+\|u_0\|_{H^{-1}(M)}\big).
\end{align}
The analogous backward result holds for \(P_q^t\), with \(r_Tu\)
prescribed in place of \(r_0u\).  In particular, the map
\eqref{eq:io-map} is well-defined and bounded for the product metric.
\end{proposition}

\begin{proof}
It is enough to prove the statement for real data and then apply it separately
to real and imaginary parts.  Thus the Riesz argument below is carried out in
the real \(L^2\) Hilbert space; after complexification, all resulting Green
identities are the complex bilinear identities fixed above.

For \(G\in L^2(Q)\), let \(z_G\in H^{2,1}(Q)\) be the unique strong
solution of
\[
 P_q^tz_G=G\quad\text{in }Q,\qquad
 z_G|_\Sigma=0,\qquad z_G(T)=0.
\]
Maximal \(L^2\) parabolic regularity and the trace theorem for $H^{2,1}(Q)$ give
\begin{equation}\label{eq:backward-regularity}
 \|z_G\|_{H^{2,1}(Q)}
 +\|\partial_\nu z_G\|_{H^{1/2,1/4}(\Sigma)}
 +\|z_G(0)\|_{H_0^1(M)}\leq C\|G\|_{L^2(Q)}.
\end{equation}
Define
\begin{align}\label{eq:transposition-functional}
 \mathscr L(G)=(F,z_G)_{L^2(Q)}
 -\langle f,\partial_\nu z_G\rangle_\Sigma
 +\langle u_0,z_G(0)\rangle_M.
\end{align}
By \eqref{eq:backward-regularity}, this is a bounded functional on
\(L^2(Q)\).  Riesz representation gives a unique \(u\in L^2(Q)\) such
that
\begin{equation}\label{eq:transposition-definition}
 (u,G)_{L^2(Q)}=\mathscr L(G),\qquad G\in L^2(Q).
\end{equation}
Taking \(G=P_q^t\zeta\), first for \(\zeta\in C_c^\infty(Q)\), shows
that \(P_qu=F\) in distributions and hence \(u\in\Hc_+(Q)\).  If
\(z\in H^{2,1}(Q)\) has zero lateral trace and \(z(T)=0\), insert
\(G=P_q^tz\) in \eqref{eq:transposition-definition} and compare with
\eqref{eq:general-green}.  One obtains
\[
 \langle\tau_0u-f,\partial_\nu z\rangle_\Sigma
 -\langle r_0u-u_0,z(0)\rangle_M=0.
\]
The trace-extension theorem permits the two test traces to be prescribed
independently, so \(\tau_0u=f\) and \(r_0u=u_0\).  The Riesz estimate,
the equation, and Proposition~\ref{prop:graph-traces} give
\eqref{eq:transposition-bound}.  If all data vanish, then
\eqref{eq:transposition-definition} gives \(u=0\), proving uniqueness.
Time reversal proves the backward statement.  Setting \(F=0\), restricting
the Neumann trace to \(\V\), and using the quotient norm prove boundedness
of the partial map.
\end{proof}

\subsection{Regularity of the difference of the input--output maps}

\begin{proposition}[Regularity of the difference map]\label{prop:difference-map}
Let $q_1,q_2\in L^\infty(Q;\mathbb R)$ and let $u_j$ be the transposition solutions with
the same input $(u_0,f)\in\X_\U$.  Then $w=u_1-u_2$ belongs to $H^{2,1}(Q)$ and
\begin{equation}\label{eq:difference-regularity}
 \|w\|_{H^{2,1}(Q)}
 \leq C\|q_1-q_2\|_{L^\infty(Q)}
       \|(u_0,f)\|_{\X_\U}.
\end{equation}
Consequently, \eqref{eq:improved-range} holds and
\begin{equation}\label{eq:difference-map-bound}
 \|\Lambda^{\U,\V}_{g,q_1}-\Lambda^{\U,\V}_{g,q_2}\|_*
 \leq C\|q_1-q_2\|_{L^\infty(Q)}.
\end{equation}
\end{proposition}

\begin{proof}
The function $w$ has zero initial and lateral Dirichlet data and satisfies
\[
 P_{q_1}w=(q_2-q_1)u_2\in L^2(Q).
\]
The transposition estimate gives
\[
 \|u_2\|_{L^2(Q)}\leq C\|(u_0,f)\|_{\X_\U}.
\]
Let \(\widetilde w\) be the strong solution of the same inhomogeneous
problem.  Maximal \(L^2\) parabolic regularity for the zero initial and
Dirichlet problem gives
\[
 \|\widetilde w\|_{H^{2,1}(Q)}
 \leq C\|(q_2-q_1)u_2\|_{L^2(Q)}.
\]
Both \(w\) and \(\widetilde w\) solve the same problem in the transposition
sense, so uniqueness of transposition solutions implies \(w=\widetilde w\).
This proves \eqref{eq:difference-regularity}.  The endpoint and normal trace
theorems for \(H^{2,1}(Q)\) yield
\[
 w(T)\in H_0^1(M),\qquad
 \partial_\nu w\in H^{1/2,1/4}(\Sigma),
\]
with norms controlled by \(\|w\|_{H^{2,1}(Q)}\).  Restriction to \(\V\)
then proves \eqref{eq:difference-map-bound}.
\end{proof}

\subsection{Conformal reduction}

We now return to the metric in \eqref{eq:cta}.  Put $\rho=(n-2)/4$ and
\begin{equation}\label{eq:conformal-potential}
 \widetilde q=cq+c^{-\rho}\Delta_{\widetilde g}(c^\rho).
\end{equation}
A direct calculation gives
\begin{equation}\label{eq:conformal-identity}
 c^{(n+2)/4}(c^{-1}\partial_t-\Delta_g+q)(c^{-\rho}v)
 = (\partial_t-\Delta_{\widetilde g}+\widetilde q)v.
\end{equation}
For completeness, put \(T_cu=c^\rho u\).  The same identity with \(q=0\)
shows that \(T_c\) is an isomorphism between the maximal graph spaces of the
original and product operators.  Hence the generalized traces for the
original equation are obtained by pullback from
Proposition~\ref{prop:graph-traces}; this does not presuppose solvability of
the original problem.  If tildes denote product-metric traces, then
\begin{align}\label{eq:conformal-traces}
 \widetilde\tau_0(T_cu)&=c^\rho\tau_0u,
 &\widetilde r_j(T_cu)&=c^\rho r_ju,\quad j\in\{0,T\},\notag\\
 \widetilde\tau_1(T_cu)&=c^{\rho+1/2}\tau_1u
 +\rho c^{-1}(\partial_{\widetilde\nu}c)
      \widetilde\tau_0(T_cu).
\end{align}
The last formula follows from
\(\nu_g=c^{-1/2}\widetilde\nu\) and the product rule.  Every
multiplication operator in \eqref{eq:conformal-traces} is an isomorphism on
the relevant negative-order trace or restriction space.  Applying
Proposition~\ref{prop:transposition} to \(v=T_cu\) therefore proves the
existence, uniqueness, and estimate for the original equation
\eqref{eq:ibvp}.  The second term in the Neumann formula depends only on the
known conformal factor and the prescribed Dirichlet trace, so it cancels in
the difference of two maps.
Thus
\begin{equation}\label{eq:map-equivalence}
 C^{-1}\|\Lambda^{\U,\V}_{c,g,q_1}-\Lambda^{\U,\V}_{c,g,q_2}\|_*
 \leq
 \|\Lambda^{\U,\V}_{\widetilde g,\widetilde q_1}
      -\Lambda^{\U,\V}_{\widetilde g,\widetilde q_2}\|_*
 \leq
 C\|\Lambda^{\U,\V}_{c,g,q_1}-\Lambda^{\U,\V}_{c,g,q_2}\|_*.
\end{equation}
Moreover, $\widetilde q_1-\widetilde q_2=c(q_1-q_2)$, so the relevant
$H^{-1}$ norms are equivalent.  It is therefore enough to prove
Theorem~\ref{thm:main} for $c=1$.  From now on $g,q$ denote the reduced product metric and potential,
and $\Lambda_{g,q}^{\U,\V}$ abbreviates $\Lambda_{1,g,q}^{\U,\V}$.
The original notation $\Lambda_{c,g,q}^{\U,\V}$ is retained only in
the unreduced main statements.

\section{Gaussian beam quasimodes with uniform estimates}\label{sec:beams}

Fix
\begin{equation}\label{eq:alpha-beta}
 \frac1{\sqrt3}<\beta<1,
 \qquad \alpha=\sqrt{1-\beta^2},
\end{equation}
and introduce the parabolic Carleman weight
\begin{equation}\label{eq:weight}
 \phi_h(t,x)=x_1+\frac{\beta^2t}{h},
 \qquad
 \Phi_h=\frac{\phi_h}{h}=\frac{x_1}{h}+\frac{\beta^2t}{h^2}.
\end{equation}
The scaled conjugated operators are
\begin{align}\label{eq:conjugated-operators}
 \mathcal L_{+,h}&=h^2e^{-\Phi_h}P_qe^{\Phi_h}
 =h^2\partial_t-h^2\Delta_g-2h\partial_{x_1}-\alpha^2+h^2q,\\
 \mathcal L_{-,h}&=h^2e^{\Phi_h}P_q^te^{-\Phi_h}
 =-h^2\partial_t-h^2\Delta_g+2h\partial_{x_1}-\alpha^2+h^2q.
\end{align}

Let $\gamma:[0,L]\to M_0$ be non-tangential and put
\begin{equation}\label{eq:beam-parameters}
 s_h=h^{-1}+i\sigma,
 \qquad \eta=\alpha^2\sigma.
\end{equation}
We construct Gaussian beam quasimodes $v_{s_h}$ for the operator
$-h^2\Delta_{g_0}-\alpha^2h^2s_h^2$ on the transversal manifold.
In the uniqueness argument it is
enough to know weak convergence of $|v_{s_h}|^2$ to the weighted arclength
measure on $\gamma$.  Stability requires a rate.  The next proposition is the
quantitative form needed below.

\begin{proposition}[Uniform estimates for Gaussian beam quasimodes]\label{prop:qgb}
Let $\G\Subset\Gamma_{\mathrm{geo}}$ satisfy Definition~\ref{def:beam-regular},
and let $\sigma$ range in a fixed compact interval.
For every $\gamma\in\G$ there are quasimodes $v_{s_h,\gamma}\in C^\infty(M_0)$
such that, uniformly in $\gamma$ and $\sigma$,
\begin{align}\label{eq:beam-residual}
 &\|v_{s_h,\gamma}\|_{L^2(M_0)}
 +\|h\nabla_{g_0}v_{s_h,\gamma}\|_{L^2(M_0)}\leq C,\\
 &\|(-h^2\Delta_{g_0}-\alpha^2h^2s_h^2)v_{s_h,\gamma}\|_{L^2(M_0)}
 \leq Ch^{3/2}.
\end{align}
They may be normalized so that for every $F\in W^{1,\infty}(M_0)$,
\begin{equation}\label{eq:quant-concentration}
 \left|
 \int_{M_0}F(x')|v_{s_h,\gamma}(x')|^2\dd V_{g_0}
 -\int_0^Le^{-2\alpha\sigma r}F(\gamma(r))\dd r
 \right|
 \leq Ch^{1/2}\|F\|_{W^{1,\infty}(M_0)}.
\end{equation}
The boundary norms of the beams are uniformly bounded on every portion of
$\partial M$ where $|\partial_\nu x_1|\geq\delta>0$.
\end{proposition}

\begin{lemma}[Local construction and smooth dependence on parameters]\label{lem:local-beams}
In a parameter chart of a family satisfying Definition~\ref{def:beam-regular},
the phase and leading
amplitude can be chosen smoothly in the incoming data. The local beam
satisfies \eqref{eq:local-beam-estimates} and the normalized diagonal
estimate \eqref{eq:single-branch-rate}, uniformly on the fixed compact
parameter support and for bounded $\sigma$.
\end{lemma}
\begin{proof}
The qualitative construction follows \cite[Section 7]{KenigSalo2013}
and \cite{DSFKLS2016}; see also \cite{Zworski2012} for the semiclassical
framework. We keep track of the uniform bounds needed here.
Put \(d=n-2\), extend \(M_0\) to a closed manifold, and extend every
\(\gamma\in\G\) a fixed distance beyond its endpoints.  The uniform geometric conditions and
compactness give a finite collection of Fermi-coordinate models
\[
 (r,y)\in I\times B(0,\delta),\qquad y\in\mathbb R^d,
\]
with the same \(\delta>0\), uniform metric bounds, and
\[
 \gamma(r)=(r,0),\qquad
 g_0^{jk}(r,0)=\delta^{jk},\qquad
 \partial_\ell g_0^{jk}(r,0)=0.
\]
Writing
\[
 g_0^{rr}(r,y)=1-\mathcal R(r)y\cdot y+O(|y|^3),
\]
let \(H\) solve the matrix Riccati equation
\[
 \dot H+H^2=\mathcal R,
 \qquad H(r_*)=H_*,\qquad \operatorname{Im}H_*>0.
\]
Take $H_*=iI_d$ at the incoming endpoint in each local smooth
orthonormal frame. Solve instead the linear system
\[
 \dot Y=Z,\quad \dot Z=\mathcal R Y,\quad Y(0)=I_d,\quad Z(0)=iI_d.
\]
Since $\mathcal R$ is real symmetric,
$Y^*Z-Z^*Y=2iI_d$ and $Y^TZ-Z^TY=0$ are conserved. The first
identity excludes a kernel of $Y$; hence $H=ZY^{-1}$ is defined on
the whole extended segment and is symmetric, and
\[
 \operatorname{Im}H=(Y^{-1})^*Y^{-1}\geq \|Y\|^{-2}I_d.
\]
Uniform bounds for $\mathcal R$ and $L_*$ give a uniform bound on
$Y,Z$ by Gronwall's inequality, and the conserved identity also bounds
$Y^{-1}$. This proves $\operatorname{Im}H\geq c_0I_d$ with one
choice of initial matrix, including in the presence of conjugate points.
The geodesic flow, parallel frame, curvature, $Y,Z,H$, and the
transport solution below depend smoothly on incoming data in every
parameter chart. Differentiating their ODEs gives uniform bounds for
each fixed finite number of parameter and arclength derivatives on
the compact parameter support. A global frame over $\Gamma$ is not
required: a finite parameter atlas suffices, and $iI_d$ is invariant
under orthogonal changes of frame.
Define
\begin{equation}\label{eq:beam-phase}
 \Theta(r,y)=\alpha\left(r+\frac12H(r)y\cdot y\right).
\end{equation}
Then, uniformly in the chosen charts,
\begin{equation}\label{eq:eikonal-order}
 \langle\dd\Theta,\dd\Theta\rangle_{g_0}-\alpha^2=O(|y|^3),
 \quad \Theta(r,0)=\alpha r,
 \quad \dd\Theta(r,0)=\alpha\dot\gamma(r)^\flat,
 \quad \operatorname{Im}\Theta\geq c|y|^2.
\end{equation}

Choose \(\chi_\perp\in C_c^\infty(B(0,\delta))\) equal to one near zero and
let \(a_0\) solve
\begin{equation}\label{eq:beam-transport}
 \dot a_0+\frac12\operatorname{tr}H\,a_0=0.
\end{equation}
The local beam is
\begin{equation}\label{eq:local-beam}
 v_{s_h}^{\rm loc}(r,y)
 =h^{-d/4}e^{is_h\Theta(r,y)}a_0(r)\chi_\perp(y).
\end{equation}
For later use, we spell out the residual calculation.  If
\(v=e^{is\Theta}a\), then
\begin{align}\label{eq:beam-residual-identity}
 &(-h^2\Delta_{g_0}-\alpha^2h^2s^2)v\notag\\
 &\quad=e^{is\Theta}\big[
 h^2s^2(\langle\dd\Theta,\dd\Theta\rangle_{g_0}-\alpha^2)a
 -ih^2s\{2\langle\dd\Theta,\dd a\rangle_{g_0}
 +(\Delta_{g_0}\Theta)a\}-h^2\Delta_{g_0}a\big].
\end{align}
Along \(y=0\), the expression in braces vanishes precisely when
\eqref{eq:beam-transport} holds.  Its Taylor expansion is therefore
\(O(|y|)\), uniformly in \(r\).  Since
\(h^2|s_h|^2=1+O(h)\) and \(h^2|s_h|=O(h)\), the eikonal and transport
contributions have sizes
\begin{align*}
 \|h^2s_h^2E_\Theta v_{s_h}^{\rm loc}\|_{L^2}
 &\leq C\||y|^3h^{-d/4}e^{-c|y|^2/h}\|_{L^2},\\
 \|h^2s_hT_\Theta(a_0\chi_\perp)e^{is_h\Theta}\|_{L^2}
 &\leq Ch\||y|h^{-d/4}e^{-c|y|^2/h}\|_{L^2},
\end{align*}
where \(E_\Theta\) and \(T_\Theta\) denote the two errors displayed in
\eqref{eq:beam-residual-identity}.  On the support of \(\dd\chi_\perp\), one
has \(\operatorname{Im}\Theta\geq c\delta^2\), so all cutoff
commutators are exponentially small.
The eikonal error in \eqref{eq:eikonal-order}, the first-order vanishing of
the transport error, and the Gaussian moment formula
\begin{equation}\label{eq:gaussian-moment}
 \big\|h^{-d/4}|y|^k e^{-c|y|^2/h}\big\|_{L^2(\mathbb R^d)}
 \leq C_kh^{k/2},\qquad k\geq0,
\end{equation}
give
\begin{align}\label{eq:local-beam-estimates}
 &\|v_{s_h}^{\rm loc}\|_{L^2}
  +\|h\nabla_{g_0}v_{s_h}^{\rm loc}\|_{L^2}\leq C,\notag\\
 &\|(-h^2\Delta_{g_0}-\alpha^2h^2s_h^2)
       v_{s_h}^{\rm loc}\|_{L^2}\leq Ch^{3/2}.
\end{align}
The three terms in \eqref{eq:beam-residual-identity} are consequently
\(O(h^{3/2})\), \(O(h^{3/2})\), and \(O(h^2)\), respectively.
Differentiating the beam once introduces at most one factor
\(h^{-1}\dd\Theta\), and hence
\(\|h\nabla v_{s_h}^{\rm loc}\|_{L^2}\leq C\).  All estimates are
uniform for \(|\sigma|\leq\sigma_0\).

\medskip\noindent\emph{Normalization.}
Write \(H=A+iB\).  Taking imaginary parts in the Riccati equation and using
Jacobi's formula shows that
\[
 \det B(r)=\det B(r_*)
 \exp\left(-2\int_{r_*}^r\operatorname{tr}A(\rho)\dd\rho\right).
\]
Together with \eqref{eq:beam-transport}, this permits the normalization
\begin{equation}\label{eq:beam-normalization}
 a_0(r_*)=\left(\frac{\alpha}{\pi}\right)^{d/4}
             \det B(r_*)^{1/4},
\end{equation}
for which the leading Gaussian mass at every \(r\) equals one.  If
\(J(r,y)\dd r\dd y=\dd V_{g_0}\), then \(J(r,0)=1\).  Hence, for a
single branch,
\begin{align}\label{eq:single-branch-rate}
 &\left|\int F(r,y)|v_{s_h}^{\rm loc}(r,y)|^2J(r,y)\dd y
       -e^{-2\alpha\sigma r}F(r,0)\right|\notag\\
 &\hspace{35mm}\leq
 Ch^{1/2}\|F\|_{W^{1,\infty}}.
\end{align}
To see this, subtract \(F(r,0)J(r,0)\), use
\(|F(r,y)J(r,y)-F(r,0)|\leq C|y|\|F\|_{W^{1,\infty}}\), and apply
\eqref{eq:gaussian-moment}.  The quadratic part of
\(e^{-2\sigma\operatorname{Re}\Theta}\), the cutoff tail, and the Taylor
remainder in the Jacobian contribute \(O(h)\) and are therefore harmless.
Integration in \(r\) preserves the \(O(h^{1/2})\) rate.

\end{proof}

\begin{lemma}[Global construction with uniform estimates]\label{lem:beam-gluing}
The local beams glue along each parameterized segment with an
$O(h^{3/2})$ residual and an $O(h^{1/2})$ change in their diagonal
concentration. The number of charts and all gluing constants are uniform.
\end{lemma}
\begin{proof}
Choose consecutive parameter intervals \(I_j\) with
\(I_j\cap I_{j+1}\ne\varnothing\), and a partition of unity
\(\{\kappa_j\}\) subordinate to them.  On an overlap, the normal Fermi
frames are related by parallel transport along \(\gamma\).  The Riccati
equation is invariant under this orthogonal change of frame.  We use the
terminal values of \(H_j\) and \(a_j\) as the initial data in the next
chart.  Thus the local phases have matching two-jets on the central curve and
the leading amplitudes have matching transported values.  In
\[
 v_{s_h,\gamma}=\sum_j\kappa_j(r)v_{s_h,j}^{\rm loc},
\]
the leading commutators containing derivatives of \(\kappa_j\) cancel
because \(\sum_j\kappa_j=1\).  The remaining mismatch has the same
\(O(|y|^3)\) eikonal and \(O(|y|)\) transport vanishing as in
\eqref{eq:beam-residual-identity}.  More explicitly, on an overlap choose
one branch \(v_{s_h,1}^{\rm loc}\) as reference and write
\[
 \sum_j[-h^2\Delta_{g_0},\kappa_j]v_{s_h,j}^{\rm loc}
 =\sum_j[-h^2\Delta_{g_0},\kappa_j]
       (v_{s_h,j}^{\rm loc}-v_{s_h,1}^{\rm loc}).
\]
In common Fermi coordinates the matching conditions give
\[
 \Theta_j-\Theta_k=O(|y|^3),\qquad a_j-a_k=O(|y|).
\]
The Gaussian moment estimate therefore yields
\[
 \|v_{s_h,j}^{\rm loc}-v_{s_h,k}^{\rm loc}\|_{L^2}=O(h^{1/2}),
 \qquad
 \|\nabla(v_{s_h,j}^{\rm loc}-v_{s_h,k}^{\rm loc})\|_{L^2}
 =O(h^{-1/2}).
\]
Since
\[
 [-h^2\Delta_{g_0},\kappa]
 =-2h^2\langle\dd\kappa,\dd\,\cdot\,\rangle_{g_0}
   -h^2(\Delta_{g_0}\kappa),
\]
the commutator contribution is $O(h^{3/2})$. Hence the global
residual retains the same order. On a fixed passage, subtracting a
reference branch gives an $O(h^{1/2})$ error in $L^2$ and an $O(1)$
bound for both beams. Consequently,
\[
 \left|\int F(|v|^2-|v_{\rm ref}|^2)\right|
 \leq\|F\|_\infty\|v-v_{\rm ref}\|_2
          (\|v\|_2+\|v_{\rm ref}\|_2)
 \leq C\sqrt h\|F\|_\infty.
\]
The leading diagonal mass is unchanged because the arclength partition
sums to one before the squared modulus is taken.

\end{proof}

\begin{lemma}[Nonstationary phase estimate for cross terms]\label{lem:beam-cross}
For distinct passages of the selected segment, the cross contribution
to concentration is bounded by $Ch^{1/2}\|F\|_{W^{1,\infty}}$.
The same bound holds for the sum of all such cross terms, uniformly
on the parameter support.
\end{lemma}
\begin{proof}
Overlapping arclength charts on one passage have already been combined
in Lemma~\ref{lem:beam-gluing}. Here we consider only different passages.
Choose a fixed arclength mesh $d_*>0$ below the common chart length and
local injectivity scale. At most
$N_*=2+\lceil2(L_*+2\epsilon_*)/d_*\rceil$ charts are needed.
Pairs at parameter distance below the local scale belong to one passage;
nonlocal returns are subject to the separation condition in
Definition~\ref{def:beam-regular}. Consequently we sum at most
$N_*(N_*-1)/2$ pairs, without counting isolated self-intersection points.

For a nonlocal pair put $\psi_{jk}=\operatorname{Re}(\Theta_j-\Theta_k)$.
On overlapping tubes with central points within $\delta_{\rm ret}$,
the central differentials of $\operatorname{Re}\Theta_j$ and
$\operatorname{Re}\Theta_k$ are $\alpha$ times the metric duals of the
corresponding oriented velocities, with the fixed $\alpha>0$ from
\eqref{eq:eikonal-order}. Their difference is therefore bounded below by
Definition~\ref{def:beam-regular}, including when the velocities are nearly
opposite. The uniform $C^2$ phase bounds then give
$|\nabla\psi_{jk}|\geq c_1>0$, after decreasing the common tube radius.
Take a smooth spatial partition supported strictly inside these coordinate
overlaps, with uniformly bounded derivatives. It is independent of $h$;
a finite atlas in position and ray parameter gives a uniform bound on its
number and derivatives. The remaining overlaps are exponentially small.
On each partition element define
\[
 X_{jk}=\frac{\nabla\psi_{jk}}{|\nabla\psi_{jk}|^2},\qquad
 X_{jk}\psi_{jk}=1,\qquad
 \|X_{jk}\|_{C^1}\leq C.
\]
The last bound follows by differentiating the quotient: its denominator
is bounded below by $c_1^2$ and the numerator and its first derivative
are controlled by the common $C^2$ phase bound.

Write the localized product as $e^{i\psi_{jk}/h}A_{jk,h}$, including
the partition and beam cutoffs in $A_{jk,h}$. Bounded attenuation factors
are also included there. If $d_j,d_k$ are transverse distances to the
two central arcs, Gaussian moments and Cauchy--Schwarz give
\[
 \|A_{jk,h}\|_{L^1}\leq C,\qquad
 \|\nabla A_{jk,h}\|_{L^1}\leq Ch^{-1/2}.
\]
To verify the second estimate, differentiation of the damping factor
produces at most $C(d_j+d_k)/h$ times that factor, since
$|\nabla\operatorname{Im}\Theta_j|\leq Cd_j$.
Each Gaussian has $L^2$ norm $O(1)$ and its product with $d_j$
has $L^2$ norm $O(\sqrt h)$. Cauchy--Schwarz therefore gives
$h^{-1}O(\sqrt h)$ in $L^1$. Derivatives of leading amplitudes,
attenuation factors and fixed cutoffs contribute $O(1)$.
This proves the claimed loss by integration, rather than by treating
the effective Gaussian support as a sharp cutoff.

Using $e^{i\psi/h}=(h/i)X_{jk}e^{i\psi/h}$ once yields
\begin{align*}
 \int_{M_0}F A_{jk,h}e^{i\psi/h}\,dV
 ={}&-\frac h i\int_{M_0}e^{i\psi/h}
       \operatorname{div}(F A_{jk,h}X_{jk})\,dV\\
 &+\frac h i\int_{\partial M_0}F A_{jk,h}e^{i\psi/h}
       \langle X_{jk},\nu\rangle\,dS.
\end{align*}
The bulk term is bounded by $Ch^{1/2}\|F\|_{W^{1,\infty}}$.
Only true segment endpoints contribute to the boundary term; chart
edges do not, because of the partition cutoffs. Uniform endpoint
transversality gives $\|v_h^{(j)}\|_{L^2(\partial M_0)}\leq C$:
restriction to the boundary is a nondegenerate transverse Gaussian
in $\dim M_0-1$ variables with normalization
$h^{-(\dim M_0-1)/4}$. Thus the boundary term is $O(h)\|F\|_\infty$.
It is exponentially small if $F$ is supported away from the endpoints.
We obtain
\begin{equation}\label{eq:cross-term-rate}
 \left|\int_{M_0}Fv_h^{(j)}\overline{v_h^{(k)}}\,dV\right|
 \leq Ch^{1/2}\|F\|_{W^{1,\infty}(M_0)}.
\end{equation}
All constants depend only on the previously fixed geometric and phase
bounds. Summing the uniformly finite number of pairs proves the lemma.
\end{proof}

\begin{proof}[Proof of Proposition~\ref{prop:qgb}]
Lemmas~\ref{lem:local-beams}--\ref{lem:beam-cross} give the interior
residual and concentration estimates. At a non-tangential endpoint,
the discrepancy between the transverse Gaussian integral over $M_0$
and the integral over a full normal fiber is bounded by
$C\exp(-c r^2/h)$, where $r$ is distance along the extended ray from
that endpoint. Integrating this bound in $r$ gives $C\sqrt h$.
Uniform endpoint transversality makes these constants uniform; this
also handles boundary terms when passage cutoffs meet the endpoints.
For compactly supported $F$ away from the endpoints these terms are
exponentially small. Finally, the boundary estimate away from the tangential set follows
from the product projection, as follows.
On \(|\partial_\nu x_1|\geq\delta\), the projection
\((x_1,x')\mapsto x'\) restricts to a local diffeomorphism from
\(\partial M\) to \(M_0\).  In the corresponding boundary graph
coordinates, its surface Jacobian is controlled by \(\delta^{-1}\), and
therefore
\[
 \|v_{s_h,\gamma}\circ\pi_0\|_{L^2(\{|\partial_\nu x_1|\geq\delta\})}
 \leq C_\delta\|v_{s_h,\gamma}\|_{L^2(M_0)}\leq C_\delta,
 \qquad \pi_0(x_1,x')=x'.
\]
A finite boundary-graph cover makes the constant uniform.  The estimate is
applied branch by branch near self-intersections, whose number is already
uniformly bounded.
\end{proof}

For $\tau\in\mathbb R$, define the parabolic quasimodes
\begin{equation}\label{eq:parabolic-quasimodes}
 V_{+,h}=e^{-i\tau t}e^{i\eta x_1}v_{s_h,\gamma}(x'),
 \qquad
 V_{-,h}=e^{i\eta x_1}\overline{v_{s_h,\gamma}(x')}.
\end{equation}
The bar is chosen so that $V_{+,h}V_{-,h}$ contains
$|v_{s_h,\gamma}|^2$.  We indicate explicitly how the frequency parameters
are matched.  Since
\[
 \alpha^2h^2s_h^2-\alpha^2
 =2i\alpha^2h\sigma-\alpha^2h^2\sigma^2,
 \qquad \eta=\alpha^2\sigma,
\]
the term \(2i\alpha^2h\sigma\) is cancelled by
\(-2h\partial_{x_1}e^{i\eta x_1}\) in \(\mathcal L_{+,h}\).  More
precisely,
\begin{align}\label{eq:parabolic-residual-expansion}
 \mathcal L_{+,h}V_{+,h}
 &=e^{-i\tau t}e^{i\eta x_1}
 \big[(-h^2\Delta_{g_0}-\alpha^2h^2s_h^2)v_{s_h,\gamma}\notag\\
 &\hspace{37mm}
 +h^2(q-i\tau-\alpha^2\beta^2\sigma^2)v_{s_h,\gamma}\big].
\end{align}
The corresponding identity for \(\mathcal L_{-,h}V_{-,h}\) is obtained by
complex conjugation, with \(s_h\) replaced by \(\overline{s_h}\); the
term \(2h\partial_{x_1}\) then supplies the required opposite
cancellation.  Equations \eqref{eq:beam-residual} and
\eqref{eq:parabolic-residual-expansion} imply
\begin{equation}\label{eq:parabolic-residual}
 \|\mathcal L_{\pm,h}V_{\pm,h}\|_{L^2(Q)}
 \leq C\big(h^{3/2}+h^2(1+|\tau|)\big),
 \qquad
 \|V_{\pm,h}\|_{L^2(Q)}\leq C.
\end{equation}
Furthermore, for $F\in W^{1,\infty}(Q)$ compactly supported after extension
in $x_1$,
\begin{align}\label{eq:space-time-concentration}
 &\left|\int_{\mathbb R}\!\int_0^T\!\int_{M_0}
 F(t,x_1,x')V_{+,h}V_{-,h}\dd V_{g_0}\dd t\dd x_1\right.\notag\\
 &\quad\left.-\int_{\mathbb R}\!\int_0^T\!\int_0^L
 e^{-i\tau t}e^{2i\alpha^2\sigma x_1}e^{-2\alpha\sigma r}
 F(t,x_1,\gamma(r))\dd r\dd t\dd x_1\right|
 \leq Ch^{1/2}\|F\|_{W^{1,\infty}}.
\end{align}

\section{Boundary Carleman estimates and geometric optics solutions}\label{sec:cgo}

We record the boundary Carleman estimate in the form used for solvability.
Its proof is included in \ref{app:carleman} to fix the powers of $h$.

We use the following subsets of the lateral boundary:
\[
 \Sigma_\pm^\circ=\{(t,x)\in\Sigma:\pm\partial_\nu x_1(x)>0\},
 \qquad
 \Sigma_{\pm,\delta_0}
 =\{(t,x)\in\Sigma:\pm\partial_\nu x_1(x)\geq\delta_0\}.
\]

\begin{proposition}[Boundary Carleman estimates]\label{prop:carleman}
There are $C,h_0>0$ such that for $0<h\leq h_0$ the following statements
hold, initially for $z\in C^\infty(\overline Q)$.
If $z|_\Sigma=0$ and $z(0)=0$, then
\begin{align}\label{eq:carleman-forward}
 &h^2\big(\|z\|_{L^2(Q)}^2+\|h\nabla_gz\|_{L^2(Q)}^2\big)
 +h^3\!\int_{\Sigma_+^\circ}
       \partial_\nu x_1|\partial_\nu z|^2\dd S_g\dd t\notag\\
 &\quad\leq
 C\Big(\|h^2e^{-\Phi_h}P_q(e^{\Phi_h}z)\|_{L^2(Q)}^2
 +h^2\|z(T)\|_{L^2(M)}^2
 +h^3\!\int_{\Sigma_-^\circ}
       |\partial_\nu x_1||\partial_\nu z|^2\dd S_g\dd t\Big).
\end{align}
If $z|_\Sigma=0$ and $z(T)=0$, then
\begin{align}\label{eq:carleman-adjoint}
 &h^2\big(\|z\|_{L^2(Q)}^2+\|h\nabla_gz\|_{L^2(Q)}^2\big)
 +h^3\!\int_{\Sigma_-^\circ}
       |\partial_\nu x_1||\partial_\nu z|^2\dd S_g\dd t\notag\\
 &\quad\leq
 C\Big(\|h^2e^{\Phi_h}P_q^t(e^{-\Phi_h}z)\|_{L^2(Q)}^2
 +h^2\|z(0)\|_{L^2(M)}^2
 +h^3\!\int_{\Sigma_+^\circ}
       \partial_\nu x_1|\partial_\nu z|^2\dd S_g\dd t\Big).
\end{align}
The estimates extend to $H^{2,1}(Q)$ functions satisfying the same
homogeneous Dirichlet and indicated temporal condition. Indeed smooth
approximation in this closed subspace follows by flattening the boundary,
odd spatial extension of the zero Dirichlet data, and smoothing; the
zero temporal trace is preserved by zero extension and one-sided time smoothing at that endpoint.
All displayed terms are continuous in $H^{2,1}$ for fixed $h$.
\end{proposition}

\begin{proposition}[Solvability with partial boundary data]
\label{prop:solvability}
For $0<h\leq h_0$ the following assertions hold.
\begin{enumerate}[label=\textup{(\roman*)}]
\item Given $F\in L^2(Q)$ and $f_-\in L^2(\Sigma_{-,\delta_0})$, there is
      $R_+\in L^2(Q)$ such that $e^{\Phi_h}R_+\in\Hc_+(Q)$,
\[
 e^{-\Phi_h}P_q(e^{\Phi_h}R_+)=F,
 \qquad R_+|_{\Sigma_{-,\delta_0}}=f_-,
\]
and
\begin{equation}\label{eq:solvability-plus}
 \|R_+\|_{L^2(Q)}
 \leq C\big(h\|F\|_{L^2(Q)}+h^{1/2}\|f_-\|_{L^2}\big).
\end{equation}
\item Given $F\in L^2(Q)$ and $f_+\in L^2(\Sigma_{+,\delta_0})$, there is
      $R_-\in L^2(Q)$ such that $e^{-\Phi_h}R_-\in\Hc_-(Q)$,
\[
 e^{\Phi_h}P_q^t(e^{-\Phi_h}R_-)=F,
 \qquad R_-|_{\Sigma_{+,\delta_0}}=f_+,
\]
and
\begin{equation}\label{eq:solvability-minus}
 \|R_-\|_{L^2(Q)}
 \leq C\big(h\|F\|_{L^2(Q)}+h^{1/2}\|f_+\|_{L^2}\big).
\end{equation}
\end{enumerate}
Put $B_\pm=h\|F\|_{L^2(Q)}+h^{1/2}\|f_\mp\|_{L^2}$,
$W_\pm=e^{\pm\Phi_h}R_\pm$, and $g_\pm=e^{\mp\Phi_h}\tau_0W_\pm$.
The solutions above can be chosen so that
\begin{align}\label{eq:explicit-strict-traces}
 \|g_+\|_{L^2(\Sigma_-^\circ)}&=\|f_-\|_{L^2},&
 \|g_+\|_{L^2(\Sigma_+^\circ)}&\leq Ch^{-1/2}B_+,\notag\\
 \|g_-\|_{L^2(\Sigma_+^\circ)}&=\|f_+\|_{L^2},&
 \|g_-\|_{L^2(\Sigma_-^\circ)}&\leq Ch^{-1/2}B_-.
\end{align}
The prescribed data are zero-extended within their same-sign sides.
Thus the opposite-side bound is explicitly
$C(h^{1/2}\|F\|+\|f_\mp\|)$.
The full traces of $W_\pm$ belong to $H^{-1/2,-1/4}(\Sigma)$;
removing conjugation in the displayed estimates on the positive and negative boundary parts costs at most
$e^{C/h^2}$. No assertion across the glancing set is needed.
Here \(R_+|_{\Sigma_{-,\delta_0}}=f_-\) means
\(e^{-\Phi_h}\tau_0(e^{\Phi_h}R_+)=f_-\), and the analogous convention is
used for \(R_-\).
\end{proposition}

\begin{proof}
We prove (i).  Since the potential and the weight are real, we first take
real \(F,f_-\), use the real Hilbert-space Riesz theorem, and then solve the
real and imaginary parts separately.  This yields the complex bilinear
formulas displayed below without a conjugation mismatch.  Let
\[
 \mathcal A_h=e^{\Phi_h}P_q^te^{-\Phi_h},
 \qquad
 \mathcal T=\{z\in C^\infty(\overline Q):z|_\Sigma=0, z(0)=z(T)=0\}.
\]
Divide \eqref{eq:carleman-adjoint} by $h^4$ and use the strict sign on
$\Sigma_{-,\delta_0}$.  For $z\in\mathcal T$ this gives
\begin{equation}\label{eq:duality-carleman}
 h^{-1}\|z\|_{L^2(Q)}+h^{-1/2}
 \|\partial_\nu z\|_{L^2(\Sigma_{-,\delta_0})}
 \leq C\left(\|\mathcal A_hz\|_{L^2(Q)}
 +h^{-1/2}\||\partial_\nu x_1|^{1/2}\partial_\nu z\|_{L^2(\Sigma_+^\circ)}
 \right).
\end{equation}
Let
\[
 \mathcal Z_hz=\left(\mathcal A_hz,
 |\partial_\nu x_1|^{1/2}\partial_\nu z|_{\Sigma_+^\circ}\right)
\]
as a map into the Hilbert space
\[
 \mathscr X_h=L^2(Q)\times L^2(\Sigma_+^\circ),
 \qquad
 \|(G,g)\|_{\mathscr X_h}^2=\|G\|_{L^2(Q)}^2+h^{-1}\|g\|_{L^2}^2.
\]
On the range of \(\mathcal Z_h\), define
\[
 \mathscr L(\mathcal Z_hz)
 =(F,z)_{L^2(Q)}-(f_-,\partial_\nu z)_{L^2(\Sigma_{-,\delta_0})}.
\]
If \(\mathcal Z_hz=0\), estimate \eqref{eq:duality-carleman} implies that
\(z=0\) and \(\partial_\nu z=0\) on \(\Sigma_{-,\delta_0}\), so the
functional is well-defined.  Moreover,
\begin{equation}\label{eq:functional-bound}
 |\mathscr L(\mathcal Z_hz)|
 \leq C\big(h\|F\|_{L^2(Q)}+h^{1/2}\|f_-\|_{L^2}\big)
       \|\mathcal Z_hz\|_{\mathscr X_h}.
\end{equation}
Hahn--Banach and the Riesz theorem give
\((R_+,S_+)\in\mathscr X_h\) representing the extension of
\(\mathscr L\), with
\begin{equation}\label{eq:riesz-components}
 \|R_+\|_{L^2(Q)}+h^{-1/2}\|S_+\|_{L^2}
 \leq C\big(h\|F\|_{L^2(Q)}+h^{1/2}\|f_-\|_{L^2}\big).
\end{equation}
In particular, for every \(z\in\mathcal T\),
\begin{align}\label{eq:riesz-identity}
 &(R_+,\mathcal A_hz)_{L^2(Q)}
 +h^{-1}(S_+,|\partial_\nu x_1|^{1/2}\partial_\nu z)_{L^2(\Sigma_+^\circ)}
 \notag\\
 &\hspace{35mm}=(F,z)_{L^2(Q)}
 -(f_-,\partial_\nu z)_{L^2(\Sigma_{-,\delta_0})}.
\end{align}

Taking \(z\in C_c^\infty(Q)\) in \eqref{eq:riesz-identity} shows that
\(e^{-\Phi_h}P_q(e^{\Phi_h}R_+)=F\) in distributions.  Thus
\(W_+=e^{\Phi_h}R_+\in\Hc_+(Q)\).  Apply the generalized Green identity to
\(W_+\) and \(e^{-\Phi_h}z\).  Comparison with
\eqref{eq:riesz-identity}, followed by
Lemma~\ref{lem:normal-density}, identifies the restrictions of
\(e^{-\Phi_h}\tau_0W_+\) to the two strict-sign open portions with
\(L^2\) functions; the full graph trace is still understood in
\(H^{-1/2,-1/4}(\Sigma)\).  Extend \(f_-\) by zero to the whole set
\(\Sigma_-^\circ\).  Independent variation of the normal traces
on the negative and positive sides gives
\[
 e^{-\Phi_h}\tau_0W_+=f_-
 \quad\text{on }\Sigma_-^\circ,
 \qquad
 e^{-\Phi_h}\tau_0W_+
 =h^{-1}|\partial_\nu x_1|^{1/2}S_+
 \quad\text{on }\Sigma_+^\circ.
\]
For precision, the comparison gives an equality of distributions tested
against every $\eta\in C_c^\infty(\Sigma_\pm^\circ)$: take a smooth
zero-Dirichlet lifting with normal trace $\eta$ and zero temporal traces.
It therefore identifies the graph trace with the displayed $L^2$
representative, rather than merely a functional on the range of
$\mathcal Z_h$. Density then gives equality on all $L^2$ tests on
each strict side. The factor $h^{-1}$ is the Riesz weight in
$\mathscr X_h$, while $\|S_+\|\leq Ch^{1/2}B_+$; together these give
the power $h^{-1/2}$ in \eqref{eq:explicit-strict-traces}.

In particular, the conjugated trace of \(R_+\) equals the prescribed
\(f_-\) on \(\Sigma_{-,\delta_0}\) and vanishes on the remaining
negative side.  The second formula and \eqref{eq:riesz-components} yield
\begin{equation}\label{eq:opposite-trace-bound}
 \|e^{-\Phi_h}\tau_0W_+\|_{L^2(\Sigma_+^\circ)}
 \leq Ch^{-1/2}\big(h\|F\|_{L^2}+h^{1/2}\|f_-\|_{L^2}\big).
\end{equation}
Estimate \eqref{eq:solvability-plus} follows from
\eqref{eq:riesz-components}.  Multiplication by \(e^{\Phi_h}\) costs at
most \(e^{C/h^2}\) on \(\overline Q\), which proves the asserted
unconjugated trace bound with only the displayed fixed polynomial loss.

For (ii), use the forward estimate \eqref{eq:carleman-forward} with
\(e^{-\Phi_h}P_qe^{\Phi_h}\) in place of \(\mathcal A_h\).  The favorable
boundary term is now on \(\Sigma_+^\circ\), so the prescribed trace
lies on \(\Sigma_{+,\delta_0}\).  The same argument proves
\eqref{eq:solvability-minus} and the stated trace bounds.
\end{proof}

\begin{proposition}[Geometric optics solutions with uniform estimates]\label{prop:cgo}
Let $q\in\Q_m$, $\gamma\in\G$, $|\sigma|\leq\sigma_0$, and
$|\tau|\leq R$.  There exist exact solutions
\begin{align}\label{eq:cgo-solutions}
 u_{+,h}&=e^{\Phi_h}(V_{+,h}+R_{+,h}),
 &P_qu_{+,h}&=0,&\supp(\tau_0u_{+,h})&\subset\U,\\
 u_{-,h}&=e^{-\Phi_h}(V_{-,h}+R_{-,h}),
 &P_q^tu_{-,h}&=0,&\supp(\tau_0u_{-,h})&\subset\V.
\end{align}
Uniformly in these parameters,
\begin{equation}\label{eq:cgo-remainder}
 \|R_{+,h}\|_{L^2(Q)}+\|R_{-,h}\|_{L^2(Q)}
 \leq C\big(h^{1/2}+h(1+R)\big).
\end{equation}
The two trace estimates are separately
\begin{align}\label{eq:individual-traces}
 \|(r_0u_{+,h},\tau_0u_{+,h})\|_{\X_\U}
 &\leq C(1+R)e^{C/h^2},\\
 \|r_Tu_{-,h}\|_{H^{-1}(M)}+
 \|\tau_0u_{-,h}\|_{\widetilde H^{-1/2,-1/4}(\V)}
 &\leq C(1+R)e^{C/h^2}.\notag
\end{align}
In particular, with $N=2$,
\begin{align}\label{eq:trace-growth}
 &\big\|(r_0u_{+,h},\tau_0u_{+,h})\big\|_{\X_\U}
 \left(\|r_Tu_{-,h}\|_{H^{-1}(M)}
 +\|\tau_0u_{-,h}\|_{\widetilde H^{-1/2,-1/4}(\V)}\right)\notag\\
 &\hspace{35mm}\leq C(1+R)^Ne^{C/h^2}.
\end{align}
\end{proposition}

\begin{proof}
Choose open neighborhoods \(\Gamma_{\mathrm D}'\), \(\Gamma_{\mathrm N}'\) such that
\[
 \partial M_+\Subset \Gamma_{\mathrm D}'\Subset \Gamma_{\mathrm D},
 \qquad \partial M_-\Subset \Gamma_{\mathrm N}'\Subset \Gamma_{\mathrm N}.
\]
After decreasing \(\delta_0\), if necessary,
\[
 \Sigma\setminus((0,T)\times \Gamma_{\mathrm D}')\subset\Sigma_{-,\delta_0},
 \qquad
 \Sigma\setminus((0,T)\times \Gamma_{\mathrm N}')\subset\Sigma_{+,\delta_0}.
\]
For the growing solution, prescribe on \(\Sigma_{-,\delta_0}\) the
piecewise \(L^2\) trace
\[
 f_{+,h}=\begin{cases}
 -V_{+,h},&\text{on }\Sigma\setminus((0,T)\times \Gamma_{\mathrm D}'),\\
 0,&\text{on }\Sigma_{-,\delta_0}\cap((0,T)\times \Gamma_{\mathrm D}').
 \end{cases}
\]
By Proposition~\ref{prop:qgb}, this trace is uniformly bounded in $L^2$.  The
source for Proposition~\ref{prop:solvability} is
\[
 -e^{-\Phi_h}P_q(e^{\Phi_h}V_{+,h})=-h^{-2}\mathcal L_{+,h}V_{+,h}.
\]
Apply part (i) of Proposition~\ref{prop:solvability} with this trace.  The
correction cancels \(V_{+,h}\) outside \((0,T)\times \Gamma_{\mathrm D}'\), so the total
Dirichlet trace is supported in $(0,T)\times\overline{\Gamma_{\mathrm D}'}\subset\U$
as a distribution on $\Sigma$. Indeed every smooth test supported outside
that closed lateral set lies in the strict negative side, where the
identified $L^2$ representative is zero. The glancing set is inside $\Gamma_{\mathrm D}'$;
it therefore creates no unexamined boundary portion in this support claim.  Moreover,
\eqref{eq:parabolic-residual} and \eqref{eq:solvability-plus} give
\[
 \|R_{+,h}\|_{L^2}
 \leq C\left[h\big(h^{-1/2}+1+R\big)+h^{1/2}\right].
\]
For the adjoint solution, prescribe \(-V_{-,h}\) on
\(\Sigma\setminus((0,T)\times \Gamma_{\mathrm N}')\) and zero on the remaining part of
\(\Sigma_{+,\delta_0}\).  Part (ii) of
Proposition~\ref{prop:solvability} proves the support condition and the same
remainder estimate.  This gives \eqref{eq:cgo-remainder}.

It remains to track the generalized traces.  Since the exact solutions satisfy
\(P_qu_{+,h}=0\) and \(P_q^tu_{-,h}=0\), their graph norms are controlled by
their \(L^2(Q)\) norms.  Hence
\[
 \|u_{\pm,h}\|_{\Hc_\pm(Q)}
 \leq Ce^{C/h^2}\big(1+h^{1/2}+h(1+R)\big).
\]
The graph trace bounds \eqref{eq:graph-traces} control the temporal and
lateral traces appearing in \eqref{eq:trace-growth}.  The \(L^2\) representatives supplied by Proposition~\ref{prop:solvability}
implement the prescribed cancellation outside \(\U\) and \(\V\); the full
lateral traces remain in the negative anisotropic Sobolev space.  For the forward input the product norm is bounded by the graph norm
of $u_{+,h}$ itself. For the transpose trace the supported-space norm
is its full global negative norm, controlled by the graph trace theorem.
Since $h\leq1$, the common factor $1+h^{1/2}+h(1+R)$ is at most
$C(1+R)$. This proves \eqref{eq:individual-traces}; multiplication
gives \eqref{eq:trace-growth} with $N=2$, without an unspecified
polynomial loss in $h$.
\end{proof}

\begin{remark}
No cutoff forcing the growing solution to vanish at $t=0$ or the decaying
solution to vanish at $t=T$ is used.  Those traces are respectively an input
and a tested output in the present input--output map.  This is why the
$h^{1/2}$ scale in \eqref{eq:cgo-remainder} replaces the weaker cutoff scale
appearing in the construction using a partial Dirichlet-to-Neumann map of \cite{ChoulliKian2018}.
\end{remark}

\section{Integral identity and estimate for the attenuated ray transform}\label{sec:ray}

Let $p=q_2-q_1$ and let $u_{+,h}$ be the growing solution for $q_1$.
Let $v_h$ solve the equation with potential $q_2$ and with the same initial
and Dirichlet data as $u_{+,h}$.  Set $w_h=u_{+,h}-v_h$.  Then
\begin{equation}\label{eq:w-equation}
 P_{q_2}w_h=pu_{+,h},
 \qquad r_0w_h=0,\qquad \tau_0w_h=0.
\end{equation}
The definition of the partial maps gives
\begin{align}\label{eq:output-difference}
 r_Tw_h&=\big(\Lambda^{\U,\V}_{g,q_1}
             -\Lambda^{\U,\V}_{g,q_2}\big)_1
             (r_0u_{+,h},\tau_0u_{+,h}),\notag\\
 \tau_1w_h|_\V&=\big(\Lambda^{\U,\V}_{g,q_1}
             -\Lambda^{\U,\V}_{g,q_2}\big)_2
             (r_0u_{+,h},\tau_0u_{+,h}).
\end{align}
By Proposition~\ref{prop:difference-map}, $w_h\in H^{2,1}(Q)$, so
\(r_Tw_h\in H_0^1(M)\) and
\(\tau_1w_h\in H^{1/2,1/4}(\Sigma)\).  Take the decaying transpose
solution $u_{-,h}$ for $q_2$.  To retain the full boundary information, apply
\eqref{eq:general-green} with first factor \(w_h\) and second factor
\(u_{-,h}\).  Its left-hand side is
\[
 (P_{q_2}w_h,u_{-,h})_{L^2(Q)}
 -(w_h,P_{q_2}^tu_{-,h})_{L^2(Q)}
 =\int_Qp\,u_{+,h}u_{-,h}\dd V_g\dd t.
\]
The complete boundary expression is
\begin{align*}
 &\langle r_Tw_h,r_Tu_{-,h}\rangle_M
 -\langle r_0w_h,r_0u_{-,h}\rangle_M\\
 &\qquad+\langle\tau_0w_h,\tau_1u_{-,h}\rangle_\Sigma
 -\langle\tau_1w_h,\tau_0u_{-,h}\rangle_\Sigma.
\end{align*}
The second and third terms vanish by \eqref{eq:w-equation}.  The first
lateral factor in the last term is a positive-order trace, while the second
is a generalized negative-order trace, so the pairing is well-defined.
Approximating the graph-space solution \(u_{-,h}\) and using
Proposition~\ref{prop:graph-traces} justifies the identity without any
additional smoothness.  We obtain
\begin{align}\label{eq:quantitative-identity}
 \int_Q p\,u_{+,h}u_{-,h}\dd V_g\dd t
 ={}&\ip{r_Tw_h}{r_Tu_{-,h}}_{H_0^1,H^{-1}}\notag\\
 &-\ip{\tau_1w_h|_\V}{\tau_0u_{-,h}}_
 {H^{1/2,1/4}(\V),\widetilde H^{-1/2,-1/4}(\V)}.
\end{align}
The initial term vanishes because \(r_0w_h=0\), and the Dirichlet trace of
\(w_h\) vanishes on all of \(\Sigma\).  The remaining lateral pairing may be
localized by Lemma~\ref{lem:local-duality}: the second factor remains
an ambient distribution supported in $\overline\V$, and is not replaced
by its negative-order restriction.  If
\begin{equation}\label{eq:epsilon}
 \eps=\|\Lambda^{\U,\V}_{g,q_1}-\Lambda^{\U,\V}_{g,q_2}\|_*,
\end{equation}
then \eqref{eq:output-difference}, the duality bounds, and
\eqref{eq:trace-growth} yield, more explicitly,
\begin{align*}
 \left|\int_Qp\,u_{+,h}u_{-,h}\dd V_g\dd t\right|
 &\leq \eps
 \|(r_0u_{+,h},\tau_0u_{+,h})\|_{\X_\U}\\
 &\quad\times\big(\|r_Tu_{-,h}\|_{H^{-1}(M)}
 +\|\tau_0u_{-,h}\|_{\widetilde H^{-1/2,-1/4}(\V)}\big),
\end{align*}
and hence
\begin{equation}\label{eq:measurement-bound}
 \left|\int_Q p\,u_{+,h}u_{-,h}\dd V_g\dd t\right|
 \leq C\eps(1+R)^Ne^{C/h^2}.
\end{equation}

Extend $p$ by zero from $Q$ to
$\mathbb R_t\times\mathbb R_{x_1}\times M_0$ and retain the notation $p$.
For almost every \(t\), the boundary agreement
\eqref{eq:boundary-agreement} gives \(p(t,\cdot)\in H_0^1(M)\).  Thus the
spatial zero extension has a uniform $H^1$ bound.  In fact, since
\(p\in W^{1,\infty}(Q)\) and has zero spatial trace, its zero extension
across \(\partial M\) belongs to \(W^{1,\infty}\) on a fixed product
neighborhood, with norm at most \(C_m\).  This observation will also justify
the use of \eqref{eq:space-time-concentration} with \(F=p\).
For every
$0<\lambda<1/2$, the
time zero extension also satisfies
\begin{equation}\label{eq:extension-regularity}
 \|p\|_{H^\lambda(\mathbb R^2_{t,x_1};L^2(M_0))}
 +\|p\|_{L^2(\mathbb R^2_{t,x_1};H^1(M_0))}\leq C_m.
\end{equation}
The restriction $\lambda<1/2$ is precisely the range in which zero extension
from a finite time interval is bounded without requiring vanishing endpoint
traces.

Define the partial Fourier transform
\begin{equation}\label{eq:partial-fourier}
 \widehat p(\tau,\xi,x')=
 \int_{\mathbb R^2}e^{-i(t\tau+x_1\xi)}p(t,x_1,x')\dd t\dd x_1.
\end{equation}

\begin{lemma}[Estimate for the Fourier and attenuated ray transforms]\label{lem:mixed-transform}
There are $C,N>0$ such that for every $\gamma\in\G$,
$|\sigma|\leq\sigma_0$, $|\tau|\leq R$, $0<h\leq h_0$, and
$1\leq R\leq h^{-1/2}$,
\begin{align}\label{eq:mixed-transform}
 &\left|
 I_{-2\alpha\sigma}
 \widehat p(\tau,-2\alpha^2\sigma,\cdot)(\gamma)
 \right|\notag\\
 &\qquad\leq C\left[
 h^{1/2}+h(1+R)+\eps(1+R)^Ne^{C/h^2}
 \right].
\end{align}
\end{lemma}

\begin{proof}
Insert the solutions of Proposition~\ref{prop:cgo} into
\eqref{eq:measurement-bound}.  Since the exponential factors
$e^{\Phi_h}$ and $e^{-\Phi_h}$ cancel, the terms containing a correction are
bounded, using \(\|p\|_{L^\infty(Q)}\leq2m\), by
\begin{align*}
 &C\big(\|V_{+,h}\|_{L^2}\|R_{-,h}\|_{L^2}
       +\|V_{-,h}\|_{L^2}\|R_{+,h}\|_{L^2}
       +\|R_{+,h}\|_{L^2}\|R_{-,h}\|_{L^2}\big)\\
 &\hspace{30mm}\leq C\big(h^{1/2}+h(1+R)\big).
\end{align*}
Indeed, if \(A_h=h^{1/2}+h(1+R)\), the three terms are bounded by
\(CA_h,CA_h,CA_h^2\), and \(A_h^2\leq CA_h\) for
\(R\leq h^{-1/2}\) and small \(h\).  The principal
term is $\int_Q pV_{+,h}V_{-,h}$.  Applying
\eqref{eq:space-time-concentration} with $F=p$ gives, up to an error
$Ch^{1/2}$,
\[
 \int_0^L e^{-2\alpha\sigma r}
 \widehat p(\tau,-2\alpha^2\sigma,\gamma(r))\dd r.
\]
This is the transform on the left of \eqref{eq:mixed-transform}.  Combining
the preceding estimates proves the claim, uniformly in \(\gamma\).  If
desired, the application to the zero extension of \(p\) can be obtained by
smooth approximation in the product neighborhood; the uniform
\(W^{1,\infty}\) bound established above passes the estimate to the limit.
\end{proof}

Choose $\sigma_0>0$ so small that $2\alpha\sigma_0\leq a_0$.  Since
$\xi=-2\alpha^2\sigma$ ranges over a fixed interval
$[-\xi_0,\xi_0]$, where $\xi_0=2\alpha^2\sigma_0$, multiplying the pointwise estimate by $\chi$ and taking the
$L^2(\Gamma,d\mu)$ norm in Lemma~\ref{lem:mixed-transform} and then applying
Assumption~\ref{ass:visibility} with \(a=-2\alpha\sigma\) imply
\begin{equation}\label{eq:low-frequency}
 \|\widehat p(\tau,\xi,\cdot)\|_{H^{-\ell}(M_0)}
 \leq C E(h,R,\eps),
 \qquad |\tau|\leq R,\quad |\xi|\leq\xi_0,
\end{equation}
where
\begin{equation}\label{eq:E}
 E(h,R,\eps)=h^{1/2}+h(1+R)
 +\eps(1+R)^Ne^{C_0/h^2}.
\end{equation}

\section{Analytic continuation and proof of stability}\label{sec:continuation}

We first give a Hilbert-valued version of the subharmonic continuation
argument used in admissible geometries; compare the quantitative analytic
continuation principle in \cite{Vessella1999}.  The advantage of the formulation
below is that the continuation exponent and its dependence on the enlarged
frequency interval are explicit.

\begin{lemma}[Quantitative analytic continuation in a Hilbert space]\label{lem:continuation}
Let $\mathcal H$ be a Hilbert space and let
$F\in H^\lambda(\mathbb R^2;\mathcal H)$, $0<\lambda<1/2$, be supported in
$[-T_0,T_0]\times[-S,S]$ in the first two variables. Fix $M>0$ and
assume explicitly that $\|F\|_{H^\lambda(\mathbb R^2;\mathcal H)}\leq M$.
Assume also that
\begin{equation}\label{eq:continuation-hyp}
 \|\widehat F(\tau,\xi)\|_{\mathcal H}\leq E,
 \qquad |\tau|\leq R,\quad |\xi|\leq\xi_0,
\end{equation}
where $0<E<e^{-1}$ and $R\geq1$.  Then
\begin{equation}\label{eq:continued-box}
 \|\widehat F(\tau,\xi+i)\|_{\mathcal H}
 \leq C E^{c/R^2},
 \qquad |\tau|\leq R,\quad |\xi|\leq R,
\end{equation}
and
\begin{equation}\label{eq:continuation-L2}
 \|F\|_{L^2(\mathbb R^2;\mathcal H)}
 \leq C\left(R E^{c/R^2}+R^{-\lambda}\right).
\end{equation}
The constants depend only on the support, the a priori
$H^\lambda$ bound, and $\xi_0$.
\end{lemma}

\begin{proof}
Compact support and Cauchy--Schwarz imply
\[
 \|\widehat F(\tau,z)\|_{\mathcal H}
 \leq Ce^{S|\operatorname{Im}z|}
       \|F\|_{L^2(\mathbb R^2;\mathcal H)}.
\]
Thus the partial Fourier transform has a Hilbert-space-valued entire representative in
\(z\), for every real \(\tau\), and all pointwise evaluations below are
well-defined.

Fix $|\tau|\leq R$ and a norm-one functional
\(\ell_0\in\mathcal H^*\).  The scalar function
\[
 z\longmapsto \ell_0(\widehat F(\tau,z))
\]
is entire and, by compact $x_1$ support, has exponential type $S$.
Let $M_*$ be a uniform a priori bound, enlarged so that $M_*\geq1$ and
$|\ell_0(\widehat F(\tau,x+iy))|\leq M_*e^{Sy}$ for $y\geq0$.
The function
\[
 G(x,y)=\log\frac{|\ell_0(\widehat F(\tau,x+iy))|}{M_*}-Sy
\]
is subharmonic in the upper half-plane, is nonpositive there, and is at most
$\log(E/M_*)\leq\log E$ on
$[-\xi_0,\xi_0]\times\{0\}$.  To justify the unbounded-domain step, let $D_A=\{z:|z|<A,\ \operatorname{Im}z>0\}$,
with $A>2\xi_0$, and let $\omega_A$ be the harmonic measure of
$[-\xi_0,\xi_0]$ in $D_A$, zero on the remaining diameter and arc.
The bounded-domain maximum principle gives $G\leq(\log E)\omega_A$.
Zeros of the scalar function cause no difficulty: apply the principle to
the subharmonic logarithm using upper limits, or truncate it from below.
Let
\[
 \omega(x,y)=\frac1\pi\int_{-\xi_0}^{\xi_0}
                  \frac{y\,ds}{(x-s)^2+y^2}.
\]
The difference $\omega-\omega_A$ is harmonic in $D_A$, vanishes on
the diameter away from the two interval endpoints, and equals $\omega$
on the arc. On that arc,
$0\leq\omega\leq 2\xi_0 A/[\pi(A-\xi_0)^2]$.
Another maximum-principle application gives the same upper bound for
$\omega-\omega_A$ throughout $D_A$. Hence $\omega_A\to\omega$
locally uniformly as $A\to\infty$, proving $G\leq(\log E)\omega$.
The Poisson kernel gives
\begin{equation}\label{eq:poisson-estimate}
 G(x,y)\leq\frac{\log E}{\pi}
 \left[\arctan\frac{x+\xi_0}{y}
       -\arctan\frac{x-\xi_0}{y}\right].
\end{equation}
For \(y=1\), the bracket has the integral representation
\[
 \arctan(x+\xi_0)-\arctan(x-\xi_0)
 =\int_{-\xi_0}^{\xi_0}\frac{\dd s}{1+(x-s)^2}.
\]
Hence it is bounded below by
\(2\xi_0/[1+(R+\xi_0)^2]\geq c/R^2\) whenever \(|x|\leq R\).
Taking the supremum over the unit ball of $\mathcal H^*$ proves
\eqref{eq:continued-box}.

Put \(G=e^{x_1}F\).  Multiplication by \(e^{x_1}\) on the fixed support is
bounded on \(H^\lambda\), and
\(\widehat G(\tau,\xi)=\widehat F(\tau,\xi+i)\).  Plancherel and
\eqref{eq:continued-box} give
\begin{align*}
 \|G\|_{L^2}^2
 &\leq C\int_{[-R,R]^2}
       \|\widehat F(\tau,\xi+i)\|_{\mathcal H}^2\dd\tau\dd\xi\\
 &\quad+C\int_{\mathbb R^2\setminus[-R,R]^2}
       \|\widehat G(\tau,\xi)\|_{\mathcal H}^2\dd\tau\dd\xi\\
 &\leq CR^2E^{2c/R^2}
 +CR^{-2\lambda}\|G\|_{H^\lambda(\mathbb R^2;\mathcal H)}^2.
\end{align*}
The last inequality uses
\(1+\tau^2+\xi^2\geq1+R^2\) outside the square.  Taking square roots and
using the equivalence between the \(L^2\) norms of \(F\) and \(G\) on the
fixed support proves \eqref{eq:continuation-L2}.
\end{proof}

\begin{proof}[Proof of Theorem~\ref{thm:main}]
By \eqref{eq:map-equivalence}, it is enough to work in the product setting.
\emph{Parameter optimization.}
Assume first that $0<\eps<\eps_0$, with $\eps_0$ fixed sufficiently small.
Choose
\begin{equation}\label{eq:parameters}
 h=\left(\frac{2C_0}{|\log\eps|}\right)^{1/2},
 \qquad R=|\log h|^{1/4},
\end{equation}
where $C_0$ is enlarged, if necessary, so that it dominates the exponential
constant in \eqref{eq:E}.  After decreasing \(\eps_0\), these parameters
satisfy \(h\leq h_0\), \(R\geq1\), and \(R\leq h^{-1/2}\), as required in
Lemma~\ref{lem:mixed-transform}.  Moreover,
\[
 \eps(1+R)^Ne^{C_0/h^2}=\eps^{1/2}(1+R)^N,
 \qquad h(1+R)=o(h^{1/2}),
\]
and hence
\begin{equation}\label{eq:E-optimized}
 E(h,R,\eps)\leq Ch^{1/2}.
\end{equation}
We apply Lemma~\ref{lem:continuation} with
$\mathcal H=H^{-\ell}(M_0)$.  The extension estimate
\eqref{eq:extension-regularity} supplies its a priori hypothesis, since
\(L^2(M_0)\hookrightarrow H^{-\ell}(M_0)\), and the zero extension has
fixed compact support in \((t,x_1)\).  By decreasing \(\eps_0\), the quantity
in \eqref{eq:E-optimized} is also smaller than \(e^{-1}\).  Since
\(R^2=|\log h|^{1/2}\) and \(|\log E|\geq c|\log h|\),
\[
 RE^{c/R^2}
 \leq C|\log h|^{1/4}
       \exp(-c|\log h|^{1/2}),
\]
which is smaller than \(|\log h|^{-N_0}\) for every fixed \(N_0>0\).
The tail term in \eqref{eq:continuation-L2} equals
\(R^{-\lambda}=|\log h|^{-\lambda/4}\).  Thus, for every
\(0<\lambda<1/2\),
\begin{equation}\label{eq:weak-final}
 \|p\|_{L^2(\mathbb R^2_{t,x_1};H^{-\ell}(M_0))}
 \leq C|\log h|^{-\lambda/4}
 \leq C|\log|\log\eps||^{-\lambda/4}.
\end{equation}

On the other hand, the a priori class and boundary agreement give
\begin{equation}\label{eq:strong-bound}
 \|p\|_{L^2(\mathbb R^2_{t,x_1};H^1(M_0))}\leq C_m.
\end{equation}
Interpolation between $H^{-\ell}(M_0)$ and $H^1(M_0)$ at the exponent that
produces $H^{-1}(M_0)$ yields
\begin{equation}\label{eq:interpolation}
 \|p\|_{L^2(\mathbb R^2;H^{-1}(M_0))}
 \leq C
 \|p\|_{L^2H^{-\ell}}^{2/(\ell+1)}
 \|p\|_{L^2H^1}^{(\ell-1)/(\ell+1)}.
\end{equation}
Moreover,
$L^2(\mathbb R^2_{t,x_1};H^{-1}(M_0))$ embeds continuously into the full
$H^{-1}$ space.  In local product coordinates this follows directly from
\[
 (1+|\zeta|^2+\tau^2+\xi^2)^{-1}
 \leq (1+|\zeta|^2)^{-1},
\]
and a finite partition of unity gives the manifold statement.  Combining
\eqref{eq:weak-final}--\eqref{eq:interpolation}
gives
\[
 \|q_1-q_2\|_{H^{-1}(Q)}
 \leq C|\log|\log\eps||^{-\lambda/[2(\ell+1)]}.
\]
Since $\lambda$ may be chosen arbitrarily close to $1/2$, every
$s<1/[4(\ell+1)]$ is admissible.  If $\eps\geq\eps_0$, the a priori bound
proves \eqref{eq:main-stability} after increasing $C$.  If $\eps=0$, replace the vanishing measurement term in the upper
bound by any $\delta>0$, run the same choice of parameters with
$\delta$ in place of $\eps$, and let $\delta\downarrow0$. The left
side is fixed and the right side tends to zero. This completes the proof.
\end{proof}

\begin{remark}[Sources of the exponent loss]
Time zero extension permits $H^\lambda$ control only for $\lambda<1/2$
without endpoint vanishing. Analytic continuation enlarges the fixed
frequency interval and introduces the second logarithm; the parameter
choice gives the weak exponent $\lambda/4$. Interpolation of
$H^{-\ell}$ with the a priori spatial $H^1$ bound contributes
$2/(\ell+1)$. Thus $s<1/[4(\ell+1)]$. These losses explain the
stated range; no optimality claim is made for this exponent.
\end{remark}

\section{Stability of the attenuated geodesic ray transform}
\label{sec:simple-proof}

We finish by verifying the ray transform stability assumption, first in the
standard simple geometry and then for a regular family of geodesics on a possibly
non-simple transversal manifold.  A smooth cutoff in ray parameter space is
kept throughout; this is essential for the normal operator to remain
pseudodifferential.

\begin{lemma}[Sobolev estimates for weighted ray transforms]
\label{lem:ray-mapping}
Consider the smooth incoming family and the nonconjugacy conditions in
Corollary~\ref{cor:regular}(i)--(ii), or a family localized by a smooth cutoff in a simple extension.
For a smooth weight $w(z,r)$ let
\[
 T_w f(z)=\chi(z)\int_0^{\tau_+(z)}w(z,r)f(\gamma_z(r))\,dr.
\]
After spatial localization inside $O$, $T_w$ extends boundedly from
$H^{-1/2}$ to $L^2(\Gamma,d\mu)$. The bound is uniform for weights
in bounded sets of sufficiently many smooth seminorms. If $w=1$,
the conormal coverage and injectivity on smooth functions supported in
$K$ imply $\|f\|_{H^{-1/2}}\leq C\|T_1f\|_{L^2}$ for supported $f$.
\end{lemma}
\begin{proof}
We specify the normal-operator input. Section~2 and Proposition~2 of
\cite{FrigyikStefanovUhlmann2008} use transverse hypersurface coordinates
for a regular family of curves. These transverse coordinates are induced
by the incoming boundary $\partial_{\mathrm{in}}SM_1$ and the unit-speed
geodesic flow. The flow and the relevant entry and exit times depend
smoothly on the incoming data. The rank condition in
Corollary~\ref{cor:regular}(ii) is their nonconjugacy condition; conormal
coverage is their Definition~1. A finite partition in ray space gives
their localized setup. Proposition~2 supplies an elliptic normal
operator of order $-1$; applying its quadratic form below gives the
$H^{-1/2}$-to-$L^2$ data estimate used here.

For completeness, nonconjugacy makes the change from direction and
nonzero travel time to the second point locally nonsingular. After a
finite partition, the kernel of $T_w^*T_w$ is therefore smooth off the
diagonal. Near the diagonal, geodesic polar coordinates give a kernel
$|x-y|^{1-d}$ times a smooth angular amplitude, where $d=\dim M_0$.
This is a classical operator of order $-1$. There is no additional
off-diagonal Fourier integral branch: the full-rank change of variables
just described accounts for every selected passage. Spatial cutoff
removes endpoint issues. The principal symbol for $w=1$ has the form
\[
 c_d\int_{S_xM_1} b(x,v)\delta(\xi(v))\,dS_x(v),
\]
where $b$ is nonnegative and contains the squared ray cutoff and the
positive density factor. Coverage and compactness of the unit cosphere
over $\overline O$ give a positive ellipticity constant. A weight that
vanishes still gives an operator of order $-1$, without asserting ellipticity.

Let $\rho\in C_c^\infty(O)$ equal one near $K$. For smooth $f$, the
mapping property of $N_w=(T_w\rho)^*(T_w\rho)$ yields
\[
 \|T_w(\rho f)\|_{L^2}^2
   =\langle N_wf,f\rangle
   \leq C\|f\|_{H^{-1/2}}^2,
 \qquad N_w:H^{-1/2}\longrightarrow H^{1/2}.
\]
Completion defines the required transform on distributions, consistently
with distributional duality. The normal construction bounds its symbol
seminorms by finitely many seminorms of $w$, so the bound is uniform.
Ellipticity and the quadratic-form G\aa rding estimate give a compact
remainder as in \eqref{eq:normal-garding}.

If the estimate without that remainder failed, take supported $f_j$
with $H^{-1/2}$ norm one and $T_1f_j\to0$ in $L^2$. A subsequence
converges weakly in $H^{-1/2}$ and strongly in $H^{-N}$, $N>1$,
to $f$. For each smooth test supported outside $K$, its pairing with
$f_j$ is zero; hence $\supp f\subset K$. Boundedness of $T_1$
gives $T_1f=0$. The elliptic parametrix then makes $f$ smooth near
$K$, and it is zero off $K$, so the assumed smooth injectivity applies.
It follows that $f=0$, contradicting the compact-remainder estimate.

Finally, replacing $d\mu$ by a specified smooth positive density
$d\widetilde\mu=\kappa\,d\mu$ changes cutoff data norms only by
\[
 (\inf_{\supp\chi}\kappa)^{1/2}\|\chi u\|_{L^2(d\mu)}
 \leq\|\chi u\|_{L^2(d\widetilde\mu)}
 \leq(\sup_{\supp\chi}\kappa)^{1/2}\|\chi u\|_{L^2(d\mu)}.
\]
These bounds are finite and positive by compactness. The smooth cutoff
remains in the amplitude; it is not treated as a positive density at its
zeros. Constants include these density bounds and required derivatives.
\end{proof}

\paragraph{Supported Sobolev norms}
For both Propositions~\ref{prop:simple-ray} and \ref{prop:regular-ray},
the following fixed-cutoff argument identifies the norm required in
Assumption~\ref{ass:visibility}. Fix
$\rho\in C_c^\infty(M_0^{\mathrm{int}})$ equal to one near $K$.
Multiplication by $\rho$, followed by restriction from $M_1$ to $M_0$
or extension by zero from $M_0$ to $M_1$, is bounded on the positive
$H^\ell$ test spaces. No boundary jump occurs because $\rho$ is
supported away from $\partial M_0$. For distributions supported in $K$,
$\langle f,v\rangle=\langle f,\rho v\rangle$; dualizing the two
bounded cutoff maps gives
\[
 C_K^{-1}\|f\|_{H^{-\ell}(M_0)}
 \leq\|f\|_{H^{-\ell}(M_1)}
 \leq C_K\|f\|_{H^{-\ell}(M_0)},\qquad\supp f\subset K.
\]
This applies to the supported realizations of these Sobolev norms,
with $C_K$ depending only on the fixed cutoff, geometry and $\ell$.

\begin{proposition}[Stability for small attenuation on a simple manifold]\label{prop:simple-ray}
Let $(M_0,g_0)$ be simple and let $K\Subset M_0^{\mathrm{int}}$.  There exist a
smooth cutoff $\chi\in C_c^\infty(\Gamma)$ whose selected segments
form a family $\G$ satisfying Definition~\ref{def:beam-regular},
a constant $a_0>0$, and, for every
$\ell>1$, a constant $C_\ell>0$ such that
\begin{equation}\label{eq:simple-visibility}
 \|f\|_{H^{-\ell}(M_0)}\leq C_\ell\|I_{a,\chi}f\|_{L^2(\Gamma,d\mu)},
 \qquad |a|\leq a_0,\quad \supp f\subset K.
\end{equation}
\end{proposition}

\begin{proof}
Extend $M_0$ slightly to a simple manifold $M_1$.  Let
\(\Gamma_K\subset\partial_{\mathrm{in}}SM_1\) be the compact set of
maximal geodesics meeting $K$.  Choose
\[
 \Gamma_K\Subset\Gamma\Subset\partial_{\mathrm{in}}SM_1,
 \qquad
 \chi\in C_c^\infty(\Gamma),\quad 0\leq\chi\leq1,
 \quad \chi=1\text{ near }\Gamma_K,
\]
so that every ray in $\supp\chi$, when restricted to $M_0$, meets
\(\partial M_0\) non-tangentially.  These restricted segments form a compact
family $\G\Subset\Gamma_{\mathrm{geo}}$.  Simplicity excludes self-intersections, and
compactness supplies all constants required in Definition~\ref{def:beam-regular}.

Keep the incoming measure $d\mu$ on $\Gamma$ as in the introduction.
Strict convexity ensures that the segment inside $M_0$ is the unique
intersection interval and hence satisfies \eqref{eq:segment-exhaustion}.
The entry time $b(z)$ is smooth on the selected family. For functions
supported near $K$, the unattenuated segment transform equals the full
transform, so its localized normal operator is
\[
 N_{0,\chi}=(I_0^{(1)})^*\chi^2I_0^{(1)}.
\]
All adjoints use the Riemannian volume and $d\mu$.
Because \(\chi=1\) on every geodesic meeting $K$, a function supported in
$K$ has zero transform on all omitted rays.  Thus cutoff data lose no
information about such a function, while the smoothness of \(\chi\) permits
the pseudodifferential calculus.

The normal calculus in Lemma~\ref{lem:ray-mapping} (see also
\cite{StefanovUhlmann2004}) gives, for some $N>1$,
\begin{equation}\label{eq:normal-garding}
 \|f\|_{H^{-1/2}(M_1)}^2\leq C\left(
 \|I_{0,\chi}f\|_{L^2(\Gamma,d\mu)}^2+
 \|f\|_{H^{-N}(M_1)}^2\right),\qquad\supp f\subset K.
\end{equation}
For a smooth supported null-vector, cutoff data equal the full data,
and scalar injectivity on a simple manifold \cite{Sharafutdinov1994}
gives $f=0$. The supported compactness and elliptic-regularity argument
of Lemma~\ref{lem:ray-mapping} therefore removes the final term:
\begin{equation}\label{eq:unattenuated-stability}
 \|f\|_{H^{-1/2}(M_1)}\leq C\|I_{0,\chi}f\|_{L^2(\Gamma,d\mu)}.
\end{equation}

It remains to quantify the attenuation perturbation.  The mean value formula
gives
\[
 (I_{a,\chi}-I_{0,\chi})f
 =a\int_0^1J_{\theta a,\chi}f\dd\theta,
 \qquad
 J_{\eta,\chi}f(\gamma)
 =\chi(\gamma)\int_0^{L_\gamma}r e^{\eta r}f(\gamma(r))\dd r.
\]
For functions supported in $K$, insert a spatial cutoff equal to one
near $K$ and use \eqref{eq:segment-exhaustion} to view $J_{\eta,\chi}$
as a smooth weighted full-ray transform. The normal operator
$J_{\eta,\chi}^*J_{\eta,\chi}$ is of order $-1$, uniformly for $|\eta|\leq1$:
the selected rays have no conjugate points and the amplitude derivatives
are uniformly bounded. Lemma~\ref{lem:ray-mapping}, with weight
$(r-b(z))e^{\eta(r-b(z))}$ for $|\eta|\leq1$, gives
\[
 \|J_{\eta,\chi}f\|_{L^2(\Gamma)}
 \leq C\|f\|_{H^{-1/2}(M_1)},\qquad |\eta|\leq1,
\]
and hence
\begin{equation}\label{eq:attenuation-perturbation}
 \|(I_{a,\chi}-I_{0,\chi})f\|_{L^2(\Gamma)}
 \leq C|a|\|f\|_{H^{-1/2}(M_1)}.
\end{equation}
Combining \eqref{eq:unattenuated-stability} and
\eqref{eq:attenuation-perturbation}, then taking $a_0$ small enough to
absorb the last term, yields
\[
 \|f\|_{H^{-1/2}(M_1)}
 \leq C\|I_{a,\chi}f\|_{L^2(\Gamma,d\mu)},\qquad |a|\leq a_0.
\]
The embedding $H^{-1/2}(M_1)\hookrightarrow H^{-\ell}(M_1)$, for every
\(\ell>1\), followed by the supported-norm equivalence above, gives
the $H^{-\ell}(M_0)$ estimate \eqref{eq:simple-visibility}.
Thus the conclusion is in exactly the space of
Assumption~\ref{ass:visibility}. Sharper mapping properties
are discussed in \cite{AssylbekovStefanov2018}.
\end{proof}

\begin{proposition}[Stability for small attenuation on a regular family]\label{prop:regular-ray}
Under the three hypotheses of Corollary~\ref{cor:regular}, there is
$a_0>0$ such that, for every $\ell>1$,
\begin{equation}\label{eq:regular-visibility}
 \|f\|_{H^{-\ell}(M_0)}
 \leq C_\ell\|I_{a,\chi}f\|_{L^2(\Gamma,d\mu)},
 \qquad |a|\leq a_0,\quad \supp f\subset K.
\end{equation}
\end{proposition}

\begin{proof}
Put $I_{0,\chi}=\chi I_0$ and
$N_{0,\chi}=I_0^*\chi^2I_0$.  The conormal coverage and absence of
conjugate points in Corollary~\ref{cor:regular}(ii) imply that
$N_{0,\chi}$ is an elliptic pseudodifferential operator of order $-1$
near $T^*K\setminus0$; this is the normal-operator construction for a
regular family in \cite[Proposition~2]{FrigyikStefanovUhlmann2008},
with its hypotheses and mapping consequences checked in
Lemma~\ref{lem:ray-mapping}. Exactly the parametrix,
G\aa rding, and compactness argument used above gives
\[
 \|f\|_{H^{-1/2}(M_1)}
 \leq C\|I_{0,\chi}f\|_{L^2(\Gamma)},\qquad \supp f\subset K.
\]
Indeed, a null-vector arising in the compactness argument is first made
smooth by ellipticity and then vanishes by the injectivity hypothesis in
Corollary~\ref{cor:regular}(iii). Uniformity is over the fixed compact
parameter support: transverse flow charts form a finite cover; the
off-diagonal exponential-map differentials have a positive least singular
value on each compact piece separated from zero travel time; near zero,
the geodesic polar expansion has uniform smooth coefficients. Coverage
gives a positive minimum of the principal symbol on the unit cosphere
over $\overline O$. Together with the fixed density and cutoff seminorms,
these bounds control the parametrix and mapping constants. Removing the
compact term uses injectivity of this fixed family; no uniform inverse
over arbitrary changes of the family is asserted.

For small $a$, the weights $e^{ar}$ form a uniformly smooth perturbation
of the unit weight on the compact selected family.  Repeating
\eqref{eq:attenuation-perturbation}, or equivalently using the perturbation
stability of regular weighted systems in
\cite{FrigyikStefanovUhlmann2008}, gives
\[
 \|(I_{a,\chi}-I_{0,\chi})f\|_{L^2(\Gamma)}
 \leq C|a|\|f\|_{H^{-1/2}(M_1)}.
\]
Absorption proves the $H^{-1/2}$ estimate uniformly for \(|a|\leq a_0\),
and Sobolev embedding first gives the corresponding $H^{-\ell}(M_1)$
bound. The supported-norm equivalence established above then proves
\eqref{eq:regular-visibility} in precisely the norm required by
Assumption~\ref{ass:visibility}. The support condition \eqref{eq:segment-exhaustion} is used here: for $\supp f\subset K$,
$I_a^{(1)}f=e^{ab}I_af$ exactly, and there are no discarded visits
to the support. This is why the full-ray calculus applies to the
segment data produced by the beams.
\end{proof}

\begin{proof}[Proof of Theorem~\ref{cor:simple} and Corollary~\ref{cor:regular}]
Propositions~\ref{prop:simple-ray} and \ref{prop:regular-ray}, respectively,
verify Assumption~\ref{ass:visibility}.  Given $s<1/8$, choose
\(\ell>1\) sufficiently close to one that
\(s<1/[4(\ell+1)]\), and apply Theorem~\ref{thm:main}.
\end{proof}

\begin{remark}[A non-simple manifold with a regular family of geodesics]
An explicit example is available when $d=\dim M_0\geq3$.
Take $M_1=\overline{B(0,4)}$, $M_0=\overline{B(0,3)}$, and
$K=\overline{B(0,1/4)}$ in $\mathbb R^d$. Keep the metric Euclidean
on the slab $|x_d|<2$. Select oriented straight chords with distance
from the origin less than $1/2$ and $|v_d|<1/8$. Such chords stay
in $|x_d|<1$, since their closest point has height below $1/2$
and their arclength from that point is at most four. Hence they
remain geodesics regardless of metric changes outside the slab.
Choose $\chi$ equal to one on chords at distance at most $1/3$
with $|v_d|\leq1/16$, and compactly supported in the indicated open family.

Every selected chord has a unique transverse intersection interval
with $M_0$, and compactness gives the uniform bounds in
Definition~\ref{def:beam-regular}. Set $O=B(0,3/10)$.
For any nonzero covector $\xi$ at $x\in\overline O$, there is a unit
vector $v$ perpendicular to both $\xi$ and $e_d$, since $d\geq3$.
The chord through $(x,v)$ has $\chi=1$, proving conormal coverage.
Injectivity follows from the Euclidean Fourier slice theorem: zero
data give vanishing of $\widehat f$ on the planes $v^\perp$ for the
open band of measured directions. For compactly supported $L^2$ data,
$\widehat f$ is continuous; the almost-everywhere statement supplied
by the slice theorem extends to every direction in this open band.
The preceding perpendicular-vector argument covers every frequency.
Thus $\widehat f=0$ and $f=0$.

To make $M_0$ non-simple, choose a small ball contained in
$B(0,3)\cap\{x_d>2\}$. Inside it embed a circle, and identify a
tubular neighborhood with a thin tube about a great circle of a
round sphere. Prescribe the pulled-back round metric on a smaller
tube and smoothly interpolate with the Euclidean metric outside the
larger tube; convex combinations of positive definite metric tensors
remain positive definite. The central circle is a closed geodesic
and stays in the round region forever, so it is trapped. To check
conjugacy explicitly, let the round sphere have radius $R_*>0$ and
parameterize the central circle by arclength $s$. For a nonzero parallel
normal field $E(s)$, the field
\[
 J(s)=\sin(s/R_*)E(s),\qquad 0\leq s\leq\pi R_*,
\]
solves the normal Jacobi equation $D_s^2J+R_*^{-2}J=0$ and vanishes
at $s=0$ and $s=\pi R_*$. The entire central half-circle, together
with an open tubular neighborhood, lies where the round metric is
unchanged. Thus its connection and curvature, and hence this nonzero
Jacobi field and its conjugate endpoints, are unaffected by the exterior
interpolation. These modifications
are disjoint from every selected chord. All three conditions of
Corollary~\ref{cor:regular} therefore hold although $M_0$ is non-simple.
This example is asserted for $d\geq3$ only.
\end{remark}

\begin{remark}
The regular-family result is genuinely non-simple: only the selected curves
covering $T^*K\setminus0$ must be free of conjugate points.  Curves outside
that family may have conjugate points or be trapped.  The regularity and stability conclusions persist under sufficiently
small smooth perturbations of the curve system, weight, and cutoff in
the finite differentiability topology required by the normal-operator estimates \cite{FrigyikStefanovUhlmann2008}.

The proof deliberately uses only small constant attenuations.  On simple
surfaces, injectivity is known for arbitrary smooth scalar attenuations
\cite{SaloUhlmann2011}.  Quantifying the dependence of the inverse on a large
attenuation could improve the controlled Euclidean frequency interval, but it
does not automatically remove the second logarithm because the stability
constant itself grows with the attenuation; compare \cite{CaroSalo2014}.
\end{remark}

\appendix
\renewcommand{\thetheorem}{\Alph{section}.\arabic{theorem}}
\section{The Carleman calculation}\label{app:carleman}

For completeness, we record the bulk and boundary calculation behind
Proposition~\ref{prop:carleman}.  It is enough to treat $q=0$ and the forward
operator.  The bounded potential is absorbed for small $h$.  The transpose estimate follows from time reversal and the product
coordinate $y_1=-x_1$, as detailed below; the convexification is performed
in that coordinate.

Fix $\delta>0$ and convexify the weight $\phi_h$ by
\begin{equation}\label{eq:convexified-weight}
 \phi_{h,\delta}=\phi_h-\frac{h}{2\delta}x_1^2.
\end{equation}
For a real-valued $v$ with $v|_\Sigma=0$ and $v(0)=0$, set
\[
 \mathcal P_\delta v
 =h^2e^{-\phi_{h,\delta}/h}(\partial_t-\Delta_g)
       e^{\phi_{h,\delta}/h}v.
\]
The general conjugation identity is
\begin{align}\label{eq:conjugation-calculation}
 &h^2e^{-\psi/h}(\partial_t-\Delta_g)e^{\psi/h}v\notag\\
 &\quad=h^2\partial_tv-h^2\Delta_gv
 -2h\langle\dd\psi,\dd v\rangle_g
 +\big(h\partial_t\psi-|\dd\psi|_g^2-h\Delta_g\psi\big)v.
\end{align}
For \(\psi=\phi_{h,\delta}\), put
\(b(x_1)=1-hx_1/\delta\).  Since the metric is a product,
\[
 \dd\psi=b\,\dd x_1,\qquad
 h\partial_t\psi=\beta^2,\qquad
 \Delta_g\psi=-\frac h\delta.
\]
Substitution into \eqref{eq:conjugation-calculation}, with the harmless
zeroth-order term \(4h^2/\delta\) assigned to the first part, gives the
following splitting.
We split $\mathcal P_\delta=\mathcal P_1+\mathcal P_2$, where
\begin{align*}
 \mathcal P_1&=h^2\partial_t
 -2h\left(1-\frac h\delta x_1\right)\partial_{x_1}
 +\frac{4h^2}{\delta},\\
 \mathcal P_2&=-h^2\Delta_g+K(x_1),\\
 K(x_1)&=-(1-\beta^2)-\frac{h^2}{\delta^2}x_1^2
 +\frac{2h}{\delta}x_1-\frac{3h^2}{\delta}.
\end{align*}
Write $(\mathcal P_1v,\mathcal P_2v)=\sum_{j=1}^6J_j$, according to the
three terms in $\mathcal P_1$ and the two terms in $\mathcal P_2$.  Integration
by parts gives
\begin{align}
 J_1&=(h^2\partial_tv,-h^2\Delta_gv)
 =\frac{h^4}{2}\|\nabla_gv(T)\|_{L^2(M)}^2,\label{eq:J1}\\
 J_2&=(h^2\partial_tv,Kv)
 =-\frac{1-\beta^2}{2}h^2\|v(T)\|_{L^2(M)}^2
 +O(h^3)\|v(T)\|_{L^2(M)}^2.\label{eq:J2}
\end{align}
Indeed, \(J_1\) follows by integrating first in space and then in time;
the initial contribution vanishes because \(v(0)=0\).  Since \(K\) is
independent of \(t\),
\[
 J_2=\frac{h^2}{2}\int_MK(x_1)|v(T,x)|^2\dd V_g,
\]
which gives \eqref{eq:J2} because
\(K=-(1-\beta^2)+O(h)\) on the bounded \(x_1\)-projection of \(M\).

Since $v=0$ on $\Sigma$,
$\nabla_gv=(\partial_\nu v)\nu$ and
$\partial_{x_1}v=(\partial_\nu x_1)\partial_\nu v$ there.  For the
first-order/Laplacian cross term, use that \(\partial_{x_1}\) is parallel
in the product metric and integrate twice by parts:
\begin{align*}
 J_3
 &=2h^3\int_Qb(\partial_{x_1}v)\Delta_gv\dd V_g\dd t\\
 &=h^3\int_\Sigma b\,\partial_\nu x_1|\partial_\nu v|^2\dd S_g\dd t
 +\frac{2h^4}{\delta}\|\partial_{x_1}v\|_{L^2(Q)}^2
 -\frac{h^4}{\delta}\|\nabla_gv\|_{L^2(Q)}^2.
\end{align*}
There is no additional boundary term involving \(v\), since its lateral
trace is zero.  Similarly,
\[
 J_4=(-2hb\partial_{x_1}v,Kv)
 =h\int_Q\partial_{x_1}(bK)|v|^2\dd V_g\dd t,
\]
whereas direct spatial integration gives
\[
 J_5=\frac{4h^4}{\delta}\|\nabla_gv\|_{L^2(Q)}^2,
 \qquad
 J_6=\frac{4h^2}{\delta}\int_QK(x_1)|v|^2\dd V_g\dd t.
\]
Expanding \(\partial_{x_1}(bK)\) and \(K\) now yields
\begin{align}
 J_3={}&h^3\int_\Sigma
 \left(1-\frac h\delta x_1\right)\partial_\nu x_1
 |\partial_\nu v|^2\dd S_g\dd t
 +\frac{2h^4}{\delta}\|\partial_{x_1}v\|_{L^2(Q)}^2
 -\frac{h^4}{\delta}\|\nabla_gv\|_{L^2(Q)}^2,
 \label{eq:J3}\\
 J_4={}&\frac{3-\beta^2}{\delta}h^2\|v\|_{L^2(Q)}^2
 -\frac{6h^3}{\delta^2}\int_Qx_1|v|^2
 +\frac{3h^4}{\delta^3}\int_Qx_1^2|v|^2
 +\frac{3h^4}{\delta^2}\|v\|_{L^2(Q)}^2,\label{eq:J4}\\
 J_5={}&\frac{4h^4}{\delta}\|\nabla_gv\|_{L^2(Q)}^2,\label{eq:J5}\\
 J_6={}&-\frac{4(1-\beta^2)}{\delta}h^2\|v\|_{L^2(Q)}^2
 +\frac{8h^3}{\delta^2}\int_Qx_1|v|^2
 -\frac{4h^4}{\delta^3}\int_Qx_1^2|v|^2
 -\frac{12h^4}{\delta^2}\|v\|_{L^2(Q)}^2.\label{eq:J6}
\end{align}
The leading bulk coefficient in $J_4+J_6$ is
\[
 (3-\beta^2)-4(1-\beta^2)=3\beta^2-1>0.
\]
Since the $x_1$ projection of $M$ is bounded, the lower-order terms are
absorbed for small $h$.  After decreasing $h_0$ so that
$1/2\leq1-hx_1/\delta\leq3/2$, we obtain
\begin{align}\label{eq:cross-lower}
 (\mathcal P_1v,\mathcal P_2v)
 &\geq -Ch^2\|v(T)\|_{L^2(M)}^2
 +ch^2\|v\|_{L^2(Q)}^2
 +ch^4\|\nabla_gv\|_{L^2(Q)}^2\notag\\
 &\quad+h^3\int_\Sigma
 \left(1-\frac h\delta x_1\right)\partial_\nu x_1
 |\partial_\nu v|^2\dd S_g\dd t.
\end{align}
Combining this with
\[
 \|\mathcal P_\delta v\|^2
 =\|\mathcal P_1v\|^2+\|\mathcal P_2v\|^2
 +2(\mathcal P_1v,\mathcal P_2v)
\]
proves the convexified estimate.  Finally, put
$z=e^{-x_1^2/(2\delta)}v$.  This fixed multiplier and its first derivatives
are bounded above and below.  Moreover,
\[
 e^{\phi_h/h}z=e^{\phi_{h,\delta}/h}v,
\]
so the conjugated operators for the convexified and original weights differ
only by this bounded multiplier, while
\[
 \partial_\nu\phi_{h,\delta}
 =\left(1-\frac h\delta x_1\right)\partial_\nu x_1.
\]
Thus the boundary signs and weighted boundary terms are uniformly comparable,
which proves \eqref{eq:carleman-forward} for \(q=0\).  For bounded \(q\),
the additional term is \(h^2qz\); its squared norm is at most
\(Ch^4\|z\|_{L^2(Q)}^2\) and is absorbed by the bulk term for small \(h\).
The complex-valued estimate follows by applying the calculation to real and
imaginary parts.  For the transpose estimate use $s=T-t$ and $y_1=-x_1$.
The forward operator in $(s,y_1,x')$ is the transpose operator in
$(t,x_1,x')$, and its linear weight satisfies
\[
 \frac{y_1}{h}+\frac{\beta^2s}{h^2}
 =-\Phi_h(t,x)+\frac{\beta^2T}{h^2}.
\]
The final constant cancels in conjugation. Apply the already proved
forward estimate on the reflected copy of $M$, with potential
$q(T-s,-y_1,x')$. Reflection is a coordinate isometry onto this copy;
no invariance of the original domain is assumed. The convexified weight
used here is $y_1+\beta^2s/h-hy_1^2/(2\delta)$, which becomes
$-\phi_h-hx_1^2/(2\delta)$ up to a constant, not
$-\phi_{h,\delta}$. The same fixed Gaussian multiplier removes this
convexification. The initial zero trace becomes $z(T)=0$, and the normal
sign reverses because $\partial_\nu y_1=-\partial_\nu x_1$.
This proves \eqref{eq:carleman-adjoint} with the required powers of $h$.

\section{Zero extension and Sobolev interpolation}\label{app:extension}

\begin{lemma}[Zero extension in Hilbert-space-valued Sobolev spaces]\label{lem:zero-extension}
Let $p\in W^{1,\infty}(Q)$ have zero spatial boundary trace. Its zero
extension $p^0$ to $\mathbb R_t\times\mathbb R_{x_1}\times M_0$ satisfies,
for every $0<\lambda<1/2$,
\[
 \|p^0\|_{H^\lambda(\mathbb R^2;L^2(M_0))}
 +\|p^0\|_{L^2(\mathbb R^2;H^1(M_0))}
 \leq C_\lambda\|p\|_{W^{1,\infty}(Q)}.
\]
No vanishing time-endpoint condition is imposed.
\end{lemma}
\begin{proof}
For almost every $t$, $p(t,\cdot)\in H_0^1(M)$. Spatial zero
extension commutes with weak first derivatives: the boundary term in
integration by parts is zero. In boundary charts the same argument for
bounded derivatives gives a $W^{1,\infty}$ spatial extension. Since
the extension operator is fixed in time, it also commutes with
$\partial_t$ on $(0,T)$. Thus the spatially extended function $v$
belongs to $H^1(0,T;L^2)$ and $L^2(0,T;H^1)$, with fixed compact
spatial support and controlled norms.

For a Hilbert-valued $v\in H^1(0,T;\mathcal H)$, write $E_0v$ for
temporal zero extension. Direct integration gives the exact cross term
in its Gagliardo seminorm:
\begin{align*}
 [E_0v]_{H^\lambda(\mathbb R;\mathcal H)}^2
 &= [v]_{H^\lambda(0,T;\mathcal H)}^2\\
 &\quad+\frac1\lambda\int_0^T\|v(t)\|_{\mathcal H}^2
       \{t^{-2\lambda}+(T-t)^{-2\lambda}\}\,dt.
\end{align*}
The first term is bounded by $C\|v\|_{H^1}^2$; the second is bounded
by the same quantity because $H^1(0,T;\mathcal H)\hookrightarrow
C([0,T];\mathcal H)$ and $2\lambda<1$. This proves the temporal
extension estimate and also explains its strict upper endpoint.
Spatial $H^1$ regularity is unchanged by temporal zero extension.
Finally,
\[
 (1+\tau^2+\xi^2)^\lambda\leq C_\lambda
 \{(1+\tau^2)^\lambda+(1+\xi^2)^\lambda\}
\]
and Plancherel combine the temporal fractional bound and the spatial
$H^1$ bound into the claimed joint bound. The possible jumps at $0,T$
are never differentiated as $L^2$ time derivatives after zero extension.
\end{proof}

\begin{lemma}[Interpolation and restriction]\label{lem:mixed-interpolation}
For functions $v$ supported in the fixed $K\Subset M_0^{\rm int}$,
\[
 \|v\|_{L^2H^{-1}}\leq C
 \|v\|_{L^2H^{-\ell}}^{2/(\ell+1)}
 \|v\|_{L^2H^1}^{(\ell-1)/(\ell+1)},\qquad \ell>1.
\]
Furthermore, restriction from the product neighborhood into $Q$ is
bounded from $L^2_{t,x_1}H^{-1}_{x'}$ to $H^{-1}(Q)$.
\end{lemma}
\begin{proof}
Use a fixed closed extension of $M_0$ and cutoffs equal to one on $K$;
the corresponding supported Sobolev norms are equivalent. Let
$\Lambda=(1-\Delta)^{1/2}$ on that extension and
$\theta=(\ell-1)/(\ell+1)$. The equality
$-1=-\ell(1-\theta)+\theta$ and H\"older's inequality in the spectral
measure of $\Lambda$, jointly with Lebesgue measure in $(t,x_1)$,
give the displayed estimate with powers $1-\theta$ and $\theta$.
This use of supported functions avoids an unqualified interpolation
identity between incompatible boundary realizations of Sobolev spaces.
For restriction, extend $\psi\in H_0^1(Q)$ by zero to the product
neighborhood. Then
\[
 |\langle v,\psi\rangle|\leq
 \|v\|_{L^2H^{-1}}\|\psi^0\|_{L^2H^1}
 \leq C\|v\|_{L^2H^{-1}}\|\psi\|_{H_0^1(Q)}.
\]
Taking the supremum over the unit ball proves the final assertion.
\end{proof}

\section*{Funding}
This work was supported by the National Natural Science Foundation of China
[grant number 12171178; funder DOI: 10.13039/501100001809].

\section*{Declaration of competing interest}
The authors declare that they have no competing interests.

\section*{Data availability}
No data were used for the research described in this article.

\section*{CRediT authorship contribution statement}
\textbf{Yujian Zheng, Zhiwen Duan and Shiqi Jing:} Conceptualization,
Methodology, Formal analysis, Validation, Writing -- original draft,
Writing -- review and editing. All three authors contributed equally.

\section*{Declaration of generative AI and AI-assisted technologies in the\texorpdfstring{\\}{ }manuscript preparation process}
During the preparation of this work, the authors used OpenAI Codex
to improve the organization, language, readability, and \LaTeX{}
presentation of the manuscript, to assist with bibliographic checks,
and to help draft and refine technical arguments. The authors are
responsible for reviewing and verifying all output, independently
checking the mathematical arguments and references, and editing the
manuscript as needed. They take full responsibility for the content
of the publication.

\end{document}